\documentclass[11pt,reqno]{amsart}
\usepackage[applemac]{inputenc}
\usepackage{bm}
\usepackage[dvipsnames,svgnames,x11names,hyperref]{xcolor}
\usepackage[hyphens]{url}
\usepackage[backref=page,colorlinks=true,linkcolor=Maroon,citecolor=blue,urlcolor=blue,hypertexnames=false,linktocpage]{hyperref}
\usepackage{amsmath}
\usepackage{amsthm}
\usepackage{xfrac}

\usepackage{fancyhdr}
\usepackage{esint}
\usepackage{enumerate}
\usepackage{pictexwd,dcpic}
\usepackage{graphicx}
\usepackage{caption}
\allowdisplaybreaks

\usepackage{indentfirst}
\usepackage{mathrsfs}
\usepackage{mathtools}
\usepackage{amssymb}
\usepackage{enumerate}
\usepackage[dvipsnames,svgnames,x11names,hyperref]{xcolor}
\usepackage{geometry}
\usepackage[hyphens]{url}
\usepackage{tikz}
\usetikzlibrary{intersections}
\usepackage{pgfplots}
\allowdisplaybreaks
\usepackage{url}
\mathtoolsset{showonlyrefs=true}
\newtheorem{theorem}{Theorem}[section]
\newtheorem{lemma}{Lemma}[section]

\newtheorem{assumption}{Assumption}[section]
\newtheorem{remark}{Remark}[section]
\newtheorem*{thmA}{Theorem A}

\numberwithin{equation}{section}

\theoremstyle{definition}

\newtheorem*{ack}{Acknowledgments}
  
\newcommand{\cA}{\mathcal{A}}
\newcommand{\cB}{\mathcal{B}}

\newcommand{\cL}{\mathcal{L}}
\newcommand{\cM}{\mathcal{M}}

\newcommand{\cO}{\mathcal{O}}

\newcommand{\cW}{\mathcal{W}}

\renewcommand{\a}{\alpha}
\renewcommand{\b}{\beta}

\newcommand{\bbH}{\mathbb H}
\newcommand{\bbS}{\mathbb S}
\newcommand{\bbB}{\mathbb B}
\newcommand{\bbR}{\mathbb R}

\newcommand{\ra}{\rightarrow}

\renewcommand{\(}{\left(}
\renewcommand{\)}{\right)}

\newcommand{\g}{\varphi}

\newcommand{\mrm}{\mathrm}

\newcommand{\tr}{\mrm{tr}}

\newcommand{\metric}[2]{\ensuremath{\langle #1, #2\rangle}}  %% command for the metric  < , >

\newcommand{\eq}[1]{\begin{equation}\begin{alignedat}{2} #1 \end{alignedat}\end{equation}}

\title[Horospherical $p$-C-M and prescribed p-shifted Weingarten curvature problems]{The horospherical $p$-Christoffel-Minkowski and prescribed $p$-shifted Weingarten curvature problems\\ in hyperbolic space}
\author[Y. Hu, H. Li, B. Xu]{Yingxiang Hu, Haizhong Li and Botong Xu}
\address{School of Mathematical Sciences, Beihang University, Beijing 100191, P.R. China}
\email{\href{mailto:huyingxiang@buaa.edu.cn}{huyingxiang@buaa.edu.cn}}

\address{Department of Mathematical Sciences, Tsinghua University, Beijing 100084, P.R. China}
\email{\href{mailto:lihz@tsinghua.edu.cn}{lihz@tsinghua.edu.cn}} 
\address{Department of Mathematics, Technion-Israel Institute of Technology, Haifa 32000, Israel}
\email{\href{mailto:botongxu@campus.technion.ac.il}{botongxu@campus.technion.ac.il}}


\begin{document}
	\setlength{\parindent}{1em}
	\maketitle
\begin{abstract}
    The $L_p$-Christoffel-Minkowski problem and the prescribed $L_p$-Weingarten curvature problem for convex hypersurfaces in Euclidean space are important problems in geometric analysis. In this paper, we consider their counterparts in hyperbolic space. For the horospherical $p$-Christoffel-Minkowski problem first introduced and studied by the second and third authors, we prove the existence of smooth, origin-symmetric, strictly horospherically convex solutions by establishing a new full rank theorem. We also propose the prescribed $p$-shifted Weingarten curvature problem and prove an existence result.
\end{abstract}

\section{Introduction}
The regular Christoffel-Minkowski problem aims to find a smooth, closed and strictly convex hypersurface in the $(n+1)$-dimensional Euclidean space $\bbR^{n+1}$ with the prescribed $k$-th area measure. Equivalently, for any given smooth positive function $f$ on $\bbS^n$, it is to find {\em strictly convex} solutions of the equation  
\eq{
p_{n-k}(D^2 h+h \sigma)=f \quad \text{on $\bbS^n$},
}
with $(D^2 h+h \sigma)>0$ on $\bbS^n$. Here $\sigma$ is the round metric on the $n$-dimensional unit sphere $\bbS^n$, $D$ its Levi-Civita connection, and $p_{n-k}$ is the normalized $(n-k)$th elementary symmetric polynomial of the eigenvalues of its argument. The solution $h$ is known as the support function of a strictly convex hypersurface in Euclidean space, and the eigenvalues of the matrix $(D^2 h+h \sigma)$ are the principal radii of curvature of the hypersurface. When $k=0$, it is the equation for the famous Minkowski problem, which has been completely resolved, see e.g. \cite{Nir53,CY76,Pog78,Lew83,Caf90}; When $k=n-1$, it is the equation for the Christoffel problem, which is also completely settled in \cite{Fir67,Ber69,LWW21}. For the intermediate Christoffel-Minkowski problem ($1\leq k \leq n-2$), under a sufficient condition on the prescribed function $f$, Guan and Ma \cite{GM03} proved the existence of a unique strictly convex solution up to translations. 

By calculating the variation of quermassintegrals of convex bodies along Firey's $p$-sums \cite{Fir62}, Lutwak \cite{Lut93} introduced the $k$-th $p$-area measure of convex bodies for $p\geq 1$. The prescribed $k$-th $p$-area measure problem for $1\leq k\leq n-1$ is called the $L_p$-Christoffel-Minkowski problem, which reduces to the classical one when $p=1$. In the regular case, this problem can be reduced to the following nonlinear PDE:
\eq{\label{Lp-CM-problem}
h^{1-p}p_{n-k}(D^2 h+h \sigma)= f \quad \text{on $\bbS^n$},
}
with $(D^2 h+h \sigma)>0$ on $\bbS^n$. For $p\geq n-k+1$, equation \eqref{Lp-CM-problem} was investigated by Hu, Ma and Shen \cite{HMS04}. For $1<p<n-k+1$, by imposing the even assumption on $f$, this equation was studied by Guan and Xia \cite{GX18}. For $0<p<1$, it was recently discussed by Bianchini, Colesanti, Pagnini and Roncoroni \cite{BCPR23}.

One of the main ingredients in the proof of the existence of convex solutions to the $L_p$-Christoffel-Minkwski problem is the constant rank theorem, which has profound implications for the geometry of solutions to equation \eqref{Lp-CM-problem}. This technique was devised to deal with the convexity property for the homotopy method of deformation by Caffarelli and Friedman \cite{CF85}, see also \cite{SWYY85}. Later, it was extended to the Christoffel-Minkowski problem and the prescribed Weingarten curvature problem for embedded hypersurfaces in Euclidean space \cite{GM03,GLM06,GMZ06}, as well as the more general fully nonlinear elliptic equations by \cite{CGM07,BG09}. The deformation lemma was also adapted to the prescribed curvature measure problems in \cite{GLM09,GLL12}. Recently, Bryan, Ivaki and Scheuer \cite{BIS23} presented a new approach to the constant rank theorem by applying the strong maximum principle to a linear differential inequality (in a viscosity sense) for the subtraces of a linear map, instead of the nonlinear test functions as in \cite{BG09}.
 
It is natural to study the analog of such prescribed measure problems for closed hypersurfaces in hyperbolic space.
A bounded domain (as well as its boundary) in hyperbolic space is called strictly horospherically convex if the principal curvatures are greater than $1$ on its boundary. The geometry of horospherically convex domains has attracted much attention in the decades, see e.g. \cite{ACW21, BM99, GRST08, HLW22, NT97, NSW22, WX14}.
 
Recently, Li and Xu \cite{LX22} introduced a summation of two sets in hyperbolic space called the hyperbolic $p$-sum, and they introduced the horospherical $k$-th $p$-area measure by calculating the variation of the $k$-th modified quermassintegrals (see \cite{ACW21}) of horospherically convex domains along the hyperbolic $p$-sum. They proposed the associated horospherical $p$-Christoffel-Minkowski problem \cite[Prob. 5.2]{LX22}. This problem aims to find strictly horospherically convex solutions to equation
\eq{ \label{s1:horo-p-CM-problem}
\varphi^{-p-k}p_{n-k}(A[\varphi])=f  \quad \text{on $\bbS^n$},
}
where $A[\varphi]$ is a symmetric $2$-tensor on $\bbS^n$ defined by 
\eq{A [\g] := D^2 \g -\frac{1}{2} \frac{|D \g|^2}{\g} \sigma +\frac{1}{2} \( \g -\frac{1}{\g}\) \sigma,
}
and a positive function $\varphi$ on $\bbS^n$ is called strictly horospherically convex if $A[\varphi]>0$ on $\bbS^n$. The strict horospherical convexity of $\varphi$ is equivalent to the fact that $\log \varphi$ is the horospherical support function of a strictly horospherically convex hypersurface. 
When $k=1$ and $p=-n$, equation \eqref{s1:horo-p-CM-problem} is the Christoffel problem in hyperbolic space proposed by \cite{EGM09}.

We first focus on the case $1 \leq k\leq n-1$ of equation \eqref{s1:horo-p-CM-problem}. By employing the volume preserving curvature flows, Li and Xu \cite[
Thm. 7.3]{LX22} proved the existence of solutions to equation \eqref{s1:horo-p-CM-problem} for $p\geq -n$, up to a constant:
\begin{thmA}[\cite{LX22}]
Let $n\geq 2$ and $1\leq k\leq n-1$ be integers, and let $p\geq -n$ be a real number. Let $f(z)$ be a smooth, positive and even function defined on $\bbS^n$. 

\begin{enumerate}[(1)] 
     \item If $p=-n$, we assume that $f$ is a positive constant function on $\bbS^n$.
    \item If $-n<p\leq -\frac{n+k}{2}$, then we assume that 
    $$
    D^2 f^{-\frac{1}{n-k}} - \frac{n-3k-2p}{n-k} |D f^{ -\frac{1}{n-k} }| \sigma + \( \frac{n+p}{n-k} \)^2 f^{ -\frac{1}{n-k} } \sigma \geq 0.
    $$
    \item If $-\frac{n+k}{2}<p<-k$, then we assume that 
    $$
    D^2  f^{-\frac{1}{n-k}} - \frac{(n-3k-2p)^2}{2(n+p)(n+k+2p)} \frac{|D  f^{-\frac{1}{n-k}} |^2}{ f^{-\frac{1}{n-k}} } \sigma + \frac{n+p}{2(n-k)}  f^{-\frac{1}{n-k}}  \sigma \geq 0.
    $$
    \item If $-k\leq p\leq n-2k$, then we assume that 
    $$
   D^2  f^{-\frac{1}{n+p}} - \frac{1}{2} \frac{|D  f^{-\frac{1}{n+p}} |^2}{ f^{-\frac{1}{n+p}} } \sigma + \frac{1}{2}  f^{-\frac{1}{n+p}} \sigma \geq 0. 
    $$
    \item If $p>n-2k$, then we assume that
    $$
    D^2  f^{-\frac{1}{n+p}} - \frac{1}{2} \frac{|D  f^{-\frac{1}{n+p}} |^2}{ f^{-\frac{1}{n+p}} } \sigma + \frac{n-k}{n+p}  f^{-\frac{1}{n+p}} \sigma \geq 0. 
    $$
\end{enumerate}
Then there exists a constant $\gamma>0$ such that the following equation admits a smooth, even and strictly horospherically convex solution,
\eq{ \label{s1:horo-p-CM-problem-gamma}
\varphi^{-p-k}p_{n-k}(A[\varphi(z)])=\gamma f(z).
}
\end{thmA}

Due to the lack of homogeneity of the horospherical $p$-area measures, it is in general difficult to remove the normalizing constant $\gamma$ in \eqref{s1:horo-p-CM-problem-gamma} via the flow approach. Recently, Chen \cite{Che24} removed the constant $\gamma$ for equation \eqref{s1:horo-p-CM-problem} with $k=n-1$ and $p=-n$ via a full rank theorem for the corresponding semilinear elliptic equation. 

In this paper, by establishing the full rank theorem (Theorem \ref{thm-full-rank}) for the fully nonlinear equation \eqref{s1:horo-p-CM-problem}, we improve Theorem A by removing the constant $\gamma$. For simplicity, we introduce the following assumptions on the function $f$ in equation \eqref{s1:horo-p-CM-problem}:
\begin{assumption}\label{Assump-f}
	Let $n \geq 2$ and $1 \leq k\leq n-1$ be integers. Let $p\geq -n$ be a real number and $f(z)$ be a smooth and positive function on $\mathbb{S}^n$.
	\begin{enumerate}
		\item If $p=-n$, then we assume that 
		\begin{align*}
		D^2 f^{-\frac{1}{n-k}} -3 |Df^{-\frac{1}{n-k}}| \sigma + \frac{f^{-\frac{1}{n-k}}}{2 + 8 \( \max_{z \in \bbS^n} f(z) \)^\frac{1}{n-k} }\sigma \geq 0.
		\end{align*}
		\item If $-n < p \leq -\frac{n+k}{2}$, then we assume that 
		\begin{align*}
		D^2 f^{-\frac{1}{n-k}} - \frac{n-3k-2p}{n-k} |D f^{ -\frac{1}{n-k} }| \sigma+ \( \frac{n+p}{n-k} \)^2 f^{ -\frac{1}{n-k} } \sigma \geq 0.
		\end{align*}
		\item If $-\frac{n+k}{2}<p < -k $, then we assume that 
		\begin{align*}
		D^2  f^{-\frac{1}{n-k}} - \frac{(n-3k-2p)^2}{2(n+p)(n+k+2p)} \frac{|D  f^{-\frac{1}{n-k}} |^2}{ f^{-\frac{1}{n-k}} } \sigma + \frac{n+p}{2(n-k)}  f^{-\frac{1}{n-k}}  \sigma \geq 0.
		\end{align*}
		\item If $-k \leq p \leq n-2k$, then we assume that 
		\begin{align*}
		D^2  f^{-\frac{1}{n+p}} - \frac{1}{2} \frac{|D  f^{-\frac{1}{n+p}} |^2}{ f^{-\frac{1}{n+p}} } \sigma + \frac{1}{2}  f^{-\frac{1}{n+p}} \sigma \geq 0. 
		\end{align*}
        \item If $p> n-2k$, then we assume that
        \begin{align*}
		D^2  f^{-\frac{1}{n+p}} - \frac{1}{2} \frac{|D  f^{-\frac{1}{n+p}} |^2}{ f^{-\frac{1}{n+p}} } \sigma + \frac{n-k}{n+p}  f^{-\frac{1}{n+p}} \sigma \geq 0. 
	\end{align*}
	\end{enumerate}
\end{assumption}

\begin{assumption}\label{Assump-barrier-f}
    Let $n \geq 1$ and $0 \leq k\leq n-1$ be integers. Let $p \geq n-2k$ be a real number and $f(z)$ be a smooth, positive function on $\bbS^n$. We assume that $0<f<2^{k-n}$ if $p=n-2k$, and $0< f <  (2k+p-n)^{ \frac{2k+p-n}{2}   } (n-k)^{n-k}/ (n+p)^{ \frac{n+p}{2} }  $ if $p>n-2k$.
\end{assumption}

\begin{theorem}\label{thm-horo-p-CM-problem}
    Let $n \geq 2$ and $1\leq k\leq n-1$ be integers, and let $p \geq -n$ be a real number.
  Assume that $f$ is a positive, even function on $\bbS^n$ that satisfies Assumption \ref{Assump-f}. If $p \geq n-2k$, we assume in addition that $f$ satisfies Assumption \ref{Assump-barrier-f}. Then the horospherical $p$-Christoffel-Minkowski problem \eqref{s1:horo-p-CM-problem} admits a smooth, even and strictly horospherically convex solution.
\end{theorem}
Let us make the following remarks on the assumption of $f$ in the above Theorem \ref{thm-horo-p-CM-problem}.
\begin{enumerate}
    \item     When $1 \leq k\leq n-1$ and $p>-n$, the convexity assumption of $f$ in Assumption \ref{Assump-f} is surprisingly the same as that of Theorem A, even though they are derived from two different ways. The former one is obtained from the full rank theorem for the fully nonlinear elliptic equation \eqref{s1:horo-p-CM-problem} that would be established in Theorem \ref{thm-full-rank}, and the latter one was derived from the pinching estimates for a volume preserving curvature flow (see \cite[Prop. 7.1]{LX22}).    When $k=n-1$ and $p=-n$, Theorem \ref{thm-horo-p-CM-problem} reduces to \cite[Theorem 1.1]{Che24}.

    \item  When $p \geq n-2k$, Assumption \ref{Assump-barrier-f} of $f$ is natural. By the work of Li-Xu \cite{LX22} and Hu-Wei-Zhou \cite{HWZ23}, the origin-centered geodesic spheres are the unique solutions to equation \eqref{s1:horo-p-CM-problem} when $p> -n$, and they are also the unique even solutions to \eqref{s1:horo-p-CM-problem} when $p=-n$. This means that, when $p \geq -n$ and $f$ is a positive constant function on $\bbS^n$, the even solution $\varphi$ to \eqref{s1:horo-p-CM-problem} must be a constant. Moreover, the argument in \cite[p. 99]{LX22} has already classified all the solutions when $f$ is constant. In particular, \eqref{s1:horo-p-CM-problem} admits no h-convex solution either $p=n-2k$ and $f \equiv c_1$ with $c_1 \geq 2^{k-n}$, or $p>2k-n$ and $f \equiv c_2$ with  $c_2>(n-k)^{n-k} (2k+p-n)^{ \frac{2k+p-n}{2}   }/ (n+p)^{ \frac{n+p}{2} }$.    
    \item When $1 \leq k\leq n-1$ and $p=-n$, the evenness assumption of $f$ in Theorem \ref{thm-horo-p-CM-problem} is natural.  When $k=n-1$ and $p=-n$, Espinar, G\'alvez and Mira \cite{EGM09} proved a Kazdan-Warner type obstruction for equation \eqref{s1:horo-p-CM-problem} by showing the relationship between this equation with the Nirenberg problem. Later, Li and Xu \cite[Thm. 9.1]{LX22} demonstrated this obstruction for the case that $0\leq k\leq n-1$ and $p=-n$. In particular, equation \eqref{s1:horo-p-CM-problem} admits no solution when $0 \leq k \leq n-1$, $p=-n$, and $f$ is a monotone rotationally symmetric function on $\bbS^n$.
\end{enumerate}

Let us also mention that the progress on the counterpart of $L_p$-Christoffel-Minkowski problem, which is the prescribed $L_p$ Weingarten curvature problem: Given a smooth positive function $f$ on $\bbS^n$, is there a closed, strictly convex hypersurface $\cM$ in $\mathbb R^{n+1}$ such that 
\eq{
p_{n-k}(\cW(\mathbf{n}^{-1}(x)))=h^{1-p}f^{-1}  \quad \text{on $\bbS^n$}\quad ?
}
Here $\mathbf{n}:\cM \ra \bbS^n$ is the Gauss map of $\cM$, and $\cW=(h_i^j)=(g^{jk}h_{ki})$ is the Weingarten matrix at the point $\mathbf{n}^{-1}(x)\in \cM$. This is equivalent to 
\eq{\label{Lp-Weingarten-curv-problem}
\frac{p_n(D^2 h+h\sigma)}{p_k(D^2 h+h\sigma)}=h^{p-1}f  \quad \text{on $\bbS^n$}.
}
When $k=0$, it is the well-known $L_p$ Minkowski problem, which was put forward by Lutwak \cite{Lut93} in his seminal work on the $L_p$ Brunn-Minkowski theory, see e.g., \cite{CW06, GLW22,HLYZ05, JLZ16,LYZ04} and a comprehensive survey by B\"or\"oczky \cite{Bor23}. For $1 \leq k\leq n-1$ and $p=1$, the existence of the solution to \eqref{Lp-Weingarten-curv-problem} was first studied by Guan and Guan \cite{GG02} under some group-invariant assumption on $f$ (e.g., $f$ is even). For $p>n-k+1$ and without the evenness assumption on $f$, the existence of a strictly convex solution to \eqref{Lp-Weingarten-curv-problem} was established through an elliptic method by \cite{GRW15}, as well as a flow approach by \cite{BIS21}. The existence of origin-symmetric strictly convex solution to \eqref{Lp-Weingarten-curv-problem} with $p=n-k+1$ and $1<p<n-k+1$ was recently proved by Lee \cite{Lee24} and Hu-Ivaki \cite{HI24}, respectively.

The following question can be considered as the natural analog of the prescribed $L_p$-Weingarten curvature problem in hyperbolic space:

{\em Given a smooth positive function $f$ on $\mathbb{S}^n$, is there a smooth, strictly horospherically convex hypersurface $\cM$ satisfying equation
\begin{equation}\label{Lp-Weingarten-eq-hyperbolic}
	f(z)~ \varphi^{(n+p)} p_{n-k} ( \cW(G^{-1}(z)) -\sigma ) =1?
\end{equation}
Here $\log \varphi:\bbS^n \ra \bbR$ is the horospherical support function of $\cM$ defined by \eqref{def-u-signed-distance}, and $G:\cM^n \ra \bbS^n$ is the horospherical Gauss map}. Equation \eqref{Lp-Weingarten-eq-hyperbolic} can be equivalently rewritten as 
\eq{
\frac{p_n(A[\varphi])}{p_k(A[\varphi])}=\varphi^{k+p}f(z) \quad \text{on $\bbS^n$}.
}
We call the above problem the prescribed $p$-shifted Weingarten curvature problem in hyperbolic space.
When $k=0$, it is the horospherical $p$-Minkowski problem proposed by Li and Xu \cite[Prob. 5.1]{LX22}, who also proved the existence of origin-symmetric solution for all $p \in (-\infty, +\infty)$ up to a constant \cite[Thm. 7.2]{LX22}. Note that
\eq{ \label{s1:Jac-G}
|\text{Jac} \ G|= \(\det A[\varphi] \)^{-1}= \varphi(z)^n p_n ( \cW(G^{-1}(z)) -\sigma ),
}
where $|\text{Jac} \ G|$ denotes the determinant of the Jacobian of the horospherical Gauss map $G$. Hence, equation \eqref{Lp-Weingarten-eq-hyperbolic} is called the prescribed horospherical surface area measure problem when $k=0$ and $p=0$. For more recent progress on the horospherical $p$-Minkowski problem, readers may refer to \cite{Che23a,LW23,LWX23}. Another special case of equation \eqref{Lp-Weingarten-eq-hyperbolic} is $p=-n$, which is called the prescribed shifted Weingarten curvature equation, was recently studied by Chen \cite{Che23b}. 

In the second part of this paper, we prove the existence result of the prescribed $p$-shifted Weingarten curvature problem \eqref{Lp-Weingarten-eq-hyperbolic} for all $0 \leq k \leq n-1$ and $p \geq -n$. 
\begin{theorem}\label{thm-Lp-Weingarten-eq-hyperbolic}
Let $n \geq 2$ and $1 \leq k \leq n-1$ be integers, and let $p \geq -n$ be a real number. Assume that $f(z)$ is a smooth, positive and even function on $\bbS^n$. If $p \geq n-2k$, we assume in addition that $f$ satisfies Assumption \ref{Assump-barrier-f}. Then the prescribed $p$-shifted Weingarten curvature problem  \eqref{Lp-Weingarten-eq-hyperbolic} admits a smooth, even and strictly horospherically convex solution. 
\end{theorem}

\begin{remark}
     When $p=-n$, Theorem \ref{thm-Lp-Weingarten-eq-hyperbolic} reduces to Theorem 1.1 in \cite{Che23b}.
\end{remark}

Just as the $L_p$ Minkowski problem  is the intersection of the $L_p$-Christoffel-Minkowski problem \eqref{Lp-CM-problem} and the prescribed Weingarten curvature problem \eqref{Lp-Weingarten-curv-problem} in Euclidean space, the horospherical $p$-Minkowski problem is  the intersection of the horospherical $p$-Christoffel-Minkowski problem \eqref{s1:horo-p-CM-problem} and the prescribed $p$-shifted Weingarten curvature problem \eqref{Lp-Weingarten-eq-hyperbolic} in hyperbolic space.
\begin{theorem}\label{thm-horo-p-Minkowski}
    Let $n \geq 2$ be an integer. Let $f(z)$ be a smooth, positive and even function on $\bbS^n$.  Assume that one of the following conditions holds:
    \begin{enumerate}
        \item $-n\leq p<n$;
        \item $p=n$, assume that $0<f<2^{-n}$;
        \item $p>n$, assume that $0<f<\frac{(p-n)^\frac{p-n}{2} n^n}{(p+n)^\frac{p+n}{2}}$.
    \end{enumerate}
    Then the horospherical $p$-Minkowski problem (i.e. $k=0$ in \eqref{s1:horo-p-CM-problem})
    \eq{
    \varphi^{-p} p_n (A[\varphi]) =f \quad \text{on $\bbS^n$}
    }
    admits a smooth, even and strictly horospherically convex solution. 
\end{theorem}

\begin{remark}
 By assuming $f$ satisfies Assumption \ref{Assump-barrier-f} when $p \geq n$, Theorem \ref{thm-horo-p-Minkowski} removes the constant $\gamma$ for the case $p \geq -n$ in \cite[Thm. 7.2]{LX22}. When $-n \leq p<n$, Theorem \ref{thm-horo-p-Minkowski} reduces to \cite[Thm. 1.3]{Che23a}.
\end{remark}

With the uniqueness result of the isotropic solution to horospherical $p$-Minkowski problem in hyperbolic plane \cite{LW24}, we establish the following existence result. 
\begin{theorem}\label{thm-horo-p-Minkowski-hyperbolic-plane}
    Let $n=1$. Let $f(z)$ be a smooth, positive and even function on $\bbS^1$. Assume that one of the following conditions holds:
    \begin{enumerate}
    \item $-7\leq p<1$;
    \item $p=1$, assume that $0<f<2^{-1}$;
    \item $p>1$, assume that $0<f<\frac{(p-1)^\frac{p-1}{2}}{(p+1)^\frac{p+1}{2}}$.
    \end{enumerate}
    Then the horospherical $p$-Minkowski problem in hyperbolic plane 
    \eq{
    \varphi^{-p}\(\varphi_{\theta\theta}-\frac{1}{2}\frac{\varphi_\theta^2}{\varphi}+\frac{\varphi-\varphi^{-1}}{2}\)=f \quad \text{on $\bbS^1$}
    }
    admits a smooth, even and strictly horospherically convex solution. 
\end{theorem}

\begin{remark}
The range of $p$ in Theorem \ref{thm-horo-p-Minkowski-hyperbolic-plane} is optimal in view of the invertiblity of the linearized operator in degree theory, see Remark \ref{remark-invertible}. 
\end{remark}

The paper is organized as follows. In Section 2, we will collect some basic properties of horospherically convex hypersurfaces and elementary symmetric polynomials. In Section 3, we will derive the a priori $C^2$ estimates for h-convex solutions to equations \eqref{s1:horo-p-CM-problem} and \eqref{Lp-Weingarten-eq-hyperbolic}. In Section 4, we will establish a deformation lemma for equation \eqref{s1:horo-p-CM-problem}, and then use it to obtain a full rank theorem for \eqref{s1:horo-p-CM-problem} in Section 5. In Section 6, we will use the degree theory for nonlinear elliptic operators to prove Theorems \ref{thm-horo-p-CM-problem}--\ref{thm-horo-p-Minkowski-hyperbolic-plane}.

\begin{ack}
Y. Hu was supported by the National Key Research and Development Program of China 2021YFA1001800, the NSFC Grant No.12101027 and the Fundamental Research Funds for the Central Universities. H. Li was supported by NSFC Grant No.12471047. The research leading to these results is part of a project that has received funding from the European Research Council (ERC) under the European Union's Horizon 2020 research and innovation programme (grant agreement No 101001677).
\end{ack}

\section{Preliminaries}
\subsection{Horospherically convex hypersurfaces}
Let $(\mathbb{S}^n, \sigma, D)$ denote the $n$-dimensional unit sphere $\mathbb{S}^n$ equipped with its canonical metric $\sigma$ and Levi-Civita connection $D$. The Minkowski space $\mathbb{R}^{n+1,1}$ is an $(n+2)$-dimensional vector space equipped with the Lorentzian metric
\eq{
	\metric{X}{Y} := \sum_{i=1}^{n+1} x_i y_i - x_{n+2}y_{n+2},
}
where $X = (x_1, \ldots, x_{n+1} , x_{n+2} )$ and $Y =(y_1, \ldots, y_{n+1} , y_{n+2} )$. 

Let $(\bbH^{n+1},\bar g,\bar\nabla)$ denote the $(n+1)$-dimensional hyperbolic space equipped with the metric $\bar g$ and Levi-Civita connection $\bar \nabla$. The hyperboloid model of the hyperbolic space $\mathbb{H}^{n+1}$ is given by
\eq{
	\mathbb{H}^{n+1} = \{ X= (x,x_{n+2}) \in \mathbb{R}^{n+1,1} ~|~ \metric{X}{X} =-1, \ x_{n+2} >0 \},
}
where $X=(x,x_{n+2})=(x_1,x_2,\cdots,x_{n+1},x_{n+2})$.  The horospheres are complete hypersurfaces in $\bbH^{n+1}$ with principal curvatures equal to $1$ everywhere, and the set of horospheres can be parameterized by $\bbS^n \times \bbR$:
	\eq{
		H_z (s) = \{X \in \bbH^{n+1} ~|~ - \metric{X}{(z,1)}  = e^s\}, \quad z \in \bbS^n, \ s \in \bbR,
	}
	where $s$ is the signed geodesic distance from the north pole $N=(0,1)$ to $H_z(s)$, and $z$ is called the center of $H_z(s)$. It is worth noting that, in the Poincar\'e ball model $\bbB^{n+1}$ of the hyperbolic space,  the horosphere $H_z(s)$ is a sphere tangential to $\partial \bbB^{n+1}$ at $z$. Define the horo-ball $B_z(s)$ enclosed by $H_z(s)$ as
	\eq{
		B_z(s) =  \{X \in \bbH^{n+1} ~|~ - \metric{X}{(z,1)}  < e^s\}, \quad z \in \bbS^n, \ s \in \bbR.
	}
Sometimes, it is also convenient to express the hyperbolic space $\bbH^{n+1}$ as the warped product $(0,\infty)\times \bbS^n$ equipped with the metric
\eq{
\bar g= dr^2 +\sinh^2 r \sigma,
}
where $r$ is the geodesic distance to the north pole $N=(0,1)$. It is known that $V=\bar\nabla \cosh r=\sinh r\partial_r$ is a conformal Killing field, i.e., $\bar\nabla(\sinh r\partial_r)=\cosh r \bar g$.
 
A bounded domain $\Omega$ (as well as its boundary $\cM = \partial \Omega$)  in $\bbH^{n+1}$ is called {\em horospherically convex} ({\em h-convex} for short) if for each $X \in \cM$, there is a horosphere enclosing $\Omega$ and touching $\Omega$ at $X$. When $\Omega$ is smooth, it is equivalent to the fact that the principal curvatures of $\cM$ are greater than or equal to $1$. We call a smooth domain $\Omega$ (or $\cM = \partial \Omega$) {\em strictly h-convex} if the principal curvatures of $\cM$ are greater than $1$. A result of Curry \cite{Cur89} shows that an h-convex, complete, immersed hypersurface in hyperbolic space is necessarily embedded and, if noncompact, it must be a horosphere. Therefore, any strictly h-convex hypersurface must be closed and hence it is uniformly h-convex, which means that all the principal curvatures exceed $1+\delta$ for some $\delta>0$. 

Now we collect some definitions and properties related to h-convex hypersurfaces in hyperbolic space. We refer to \cite[Sect. 5]{ACW21}, \cite[Sect. 2]{LX22}, and \cite{EGM09} for details.
Let $(\cM, g)$ be a smooth, strictly h-convex hypersurface in $\mathbb{H}^{n+1}$. Denote by $\nu$ the unit outward normal of $\cM$. Then the {\em horospherical Gauss map} $G: \cM \to \mathbb{S}^n$ is defined by the unique point $z = G(X) \in \mathbb{S}^n$ such that 
\eq{
	X-\nu =e^{-\lambda} (z,1),
}
where $\lambda \in \bbR$ is the signed geodesic distance from $N=(0,1)$ to the horosphere tangential to $\cM$ at $X$ in $\bbH^{n+1}$. Note that the uniform h-convexity of $\cM$ implies that the horospherical Gauss map $G$ is a diffeomorphism. The {\em horospherical support function} of $\cM$ is then defined by 
\eq{\label{def-u-signed-distance}
	u(z) := \log  (-\metric{G^{-1}(z)}{(z,1)}),  \quad z \in \mathbb{S}^n.
}
Then $\lambda (G^{-1} (z)) = u(z)$.  For convenience, we set $\g(z) := e^{u(z)}$.  Denote by $h_{ij}$ the second fundamental form of $\cM$, and let $h_i{}^j = h_{il} g^{lj}$. Then
\begin{equation}\label{shifted principal curv-support}
	\( h_i{}^j (G^{-1} (z)) -\delta_i{}^j\) A_{jl} [\g (z)] = \frac{1}{\g (z)} \sigma_{il},  \quad \forall z \in \mathbb{S}^n, 
\end{equation}
where $A[\g]$ is a symmetric $2$-tensor on $\mathbb{S}^n$ defined by
\begin{equation}\label{def-Aij}
	A [\g] := D^2 \g -\frac{1}{2} \frac{|D \g|^2}{\g} \sigma +\frac{1}{2} \( \g -\frac{1}{\g}\) \sigma.
\end{equation}
Hence, the strict h-convexity of $\cM$ implies that $A[\g]$ is positive definite on $\bbS^n$. Conversely, for any positive $\g \in C^2(\mathbb{S}^n)$ with $A [\g]>0$, we can recover the hypersurface $\cM$ in $\bbH^{n+1}$ through the embedding 
\eq{\label{X(z)}
	\bbS^n \ni z \mapsto \frac{\g}{2}(-z,1) + \(\frac{|D \g|^2}{2\g} +\frac{1}{2 \g} \) (z,1) - (D \g, 0) \in \bbH^{n+1}. 
}

Denote by $\kappa = (\kappa_1, \ldots, \kappa_n)$ the principal curvatures of $\cM$, which are the eigenvalues of the Weingarten matrix $\(h_i{}^j\)$. Then the shifted principal curvature $\tilde{\kappa}=(\tilde \kappa_1,\ldots,\tilde\kappa_n)=(\kappa_1-1,\ldots,\kappa_n-1)$ are the eigenvalues of the matrix $\(h_i{}^j-\delta_i{}^j\)$. In view of \eqref{shifted principal curv-support}, the matrix $\(h_i{}^j(G^{-1}(z))-\delta_i{}^j\)$ is the inverse matrix of $\varphi(z)\sigma^{jl}A_{li}[\varphi(z)]$.

Now, we derive some useful formulas for $A[\g]$. For simplicity, we take 
\eq{\label{s5:psi}
\psi=\frac{1}{2}\frac{|D\varphi|^2}{\varphi}+\frac{1}{2}\(\varphi+\frac{1}{\varphi}\).
}
Then we have
\eq{\label{s5:rel-A-phi-psi}
A_{ij}[\varphi]=\varphi_{ij}+(\varphi-\psi)\sigma_{ij}.
}

In the following, we take an orthonormal basis and assume that $A_{ij}= A_{ij}[\varphi]$ is diagonal at the point where we calculate; that is, $\sigma_{ij}=\delta_{ij}$ and $A_{ij}=\lambda_i \delta_{ij}$.
\begin{lemma}\label{s5:lem-not codazzi}
    We have
\eq{\label{s5:psi_i}
\psi_i &=\frac{\varphi_i}{\varphi}A_{ii}=\frac{\lambda_i \varphi_i}{\varphi}, 
}
\eq{\label{s5:A-Codazzi} 
A_{ij,l}-A_{il,j} = -\frac{\g_l}{\g} A_{ll} \delta_{ij}+ \frac{\g_j}{\g} A_{jj} \delta_{il},
}
\eq{\label{s5:psi-ii} 
\psi_{ii}&=\frac{A_{ii}^2}{\varphi}+\(\psi-\varphi\)\frac{A_{ii}}{\varphi}-2\frac{\varphi_i^2}{\varphi^2}A_{ii}+\sum_l \frac{\varphi_{l}}{\varphi}A_{ii,l}+\sum_{l}\frac{\varphi_{l}^2}{\varphi^2}A_{ll}.
}
and
\eq{ \label{s5:Aiiaa-Aaaii}
A_{ii,\alpha\alpha}-A_{\alpha\alpha,ii}=&\sum_l\frac{\varphi_{l}}{\varphi}(A_{ii,l}-A_{\alpha\alpha,l})+\frac{1}{\varphi}(A_{ii}^2-A_{\alpha\alpha}^2)\\
&+\frac{\psi}{\varphi}(A_{ii}-A_{\alpha\alpha})-\frac{2}{\varphi^2}(A_{ii}\varphi_i^2-A_{\alpha\alpha}\varphi_\alpha^2).
}
\end{lemma}
	\begin{proof}
		Using the definition of $\psi$ in \eqref{s5:psi}, we have
		\eq{
		 \psi_i & = \frac{ \g_l \g_{li} }{\g} - \frac{1}{2} \frac{|D \g|^2}{\g^2} \g_i + \frac{1}{2} \(\g_i -\frac{1}{\g^2} \g_i\) \\
			& = \frac{\g_l}{\g} \( \g_{li} - \frac{1}{2} \frac{|D \g|^2}{\g} \delta_{li} + \frac{1}{2} ( \g - \frac{1}{\g} ) \delta_{li} \)\\
			& = \frac{\g_l}{\g} A_{li} = \frac{\g_i}{\g} A_{ii}=\frac{\lambda_i \varphi_i}{\varphi}.
		}
		Thus we obtain \eqref{s5:psi_i}. Then formula \eqref{s5:A-Codazzi} follows from \eqref{s5:psi_i} and the well-known fact that $D(D^2 \varphi +\varphi \sigma)$ is a Codazzi tensor on $\mathbb{S}^n$, 
		\eq{
			A_{ij,l} -A_{il,j} &= D_l (\varphi_{ij} +\varphi \delta_{ij}) - \psi_l \delta_{ij} -D_j (\varphi_{il} +\varphi \delta_{il}) + \psi_j \delta_{il}\\
			&= - \psi_l \delta_{ij}+ \psi_j \delta_{il} \\
			&= -\frac{\varphi_l}{\varphi} A_{ll} \delta_{ij} + \frac{\varphi_j}{\varphi} A_{jj} \delta_{il}.
		}
        Using \eqref{s5:psi_i}, for fixed index $i$ we have,
		\eq{\label{s5:psi-ii-1}
			\psi_{ii} &= D_i \( \frac{\varphi_l}{\varphi} A_{li} \) = \frac{\varphi_{li}}{\varphi} A_{li} - \frac{\varphi_l \varphi_i A_{li} }{\varphi^2} + \frac{\varphi_l}{\varphi} A_{li,i}  \\
			&= \frac{\varphi_{ii}}{\varphi} A_{ii} - \frac{\varphi_i^2}{\varphi^2} A_{ii} + \sum_l \frac{\varphi_l}{\varphi} A_{li,i}. 
		}
		By \eqref{s5:rel-A-phi-psi}, we have
		\eq{
			\frac{\varphi_{ii}}{\varphi} A_{ii} = \frac{A_{ii}^2}{\varphi} + \( \frac{\psi}{\varphi} -1\) A_{ii}.
		}
		On the other hand, it follows from \eqref{s5:A-Codazzi} that
		\eq{
			\sum_l \frac{\varphi_l}{\varphi} A_{li,i}  &= \sum_l \frac{\varphi_l}{\varphi} \( A_{ii,l} - \frac{\varphi_i}{\varphi} A_{ii} \delta_{li} + \frac{\varphi_l}{\varphi} A_{ll} \)\\
			&=\sum_l \frac{\varphi_l A_{ii,l}}{\varphi} -\frac{\varphi_i^2}{\varphi^2} A_{ii} + \sum_l \frac{\varphi_l^2}{\varphi^2} A_{ll}.
		}
		Inserting the above two formulas into \eqref{s5:psi-ii-1}, we obtain \eqref{s5:psi-ii}.
        Note that 
		\eq{
			A_{ii, \a\a} &=  \varphi_{ii, \a\a} + \varphi_{\a\a} - \psi_{\a\a}, \\
			A_{\a\a, ii} &= \varphi_{\a\a, ii}+\varphi_{ii} -  \psi_{ii}.
		}
		Using the Ricci identity for tensors on $\mathbb{S}^n$, we have
		\eq{
			\varphi_{ii,\a\a} - \varphi_{\a\a, ii} = 2\varphi_{ii} -2 \varphi_{\a\a},
		}
		and thus
		\eq{
			A_{ii,\a\a} - A_{\a\a, ii} &= ( \varphi_{ii, \a\a} - \varphi_{\a\a, ii} ) + ( \varphi_{\a\a} - \varphi_{ii} ) +\psi_{ii} -  \psi_{\a\a}\\
			&= \varphi_{ii} -\varphi_{\a\a} + \psi_{ii} -\psi_{\a\a}\\
            &= A_{ii}-A_{\a\a} + \psi_{ii}- \psi_{\a\a}.
		}
		The formula \eqref{s5:Aiiaa-Aaaii} follows by inserting \eqref{s5:psi-ii} into the above formula. This completes the proof of Lemma \ref{s5:lem-not codazzi}.
		\end{proof}

\subsection{Elementary symmetric polynomials}
The $k$th elementary symmetric polynomial $S_k:\mathbb{R}^n \to \mathbb{R}$ is defined by
\eq{
S_k(\lambda):=\sum_{1 \leq i_1 < \cdots <i_k \leq n} \lambda_{i_1} \cdots \lambda_{i_k}, \quad k=1,2,\ldots,n.
}
We also set $S_0=1$ and $S_{k}=0$ if either $k<0$ or $k>n$ by convention. This definition of $S_k$ can be extended to symmetric matrices $A$ as follows (see \cite[Prop. 2.2]{GM03}): Let $A=(A_{ij})$ be an $(n\times n)$ symmetric matrix, $S_k$ can be defined as 
\eq{
S_k(A) &:=\frac{1}{k!}\sum_{\substack{i_1,\ldots,i_k=1\\j_1,\ldots,j_k=1}}^{n} \delta_{i_1\ldots i_k}^{j_1 \ldots j_k}A_{i_1 j_1}\ldots A_{i_k j_k}.
}
where $\delta_I^J$ for the indices $I=(i_1,\ldots,i_m)$ and $J=(j_1,\ldots,j_m)$ is defined as 
\eq{
\delta_I^J=\left\{\begin{aligned}
    &1, \quad  &\text{if $I$ is an even permutation of $J$};\\
    &-1, \quad &\text{if $I$ is an odd permutation of $J$};\\
    &0, \quad &\text{otherwise}.
\end{aligned}\right.
}
When $A$ is diagonal, then $S_k(A)=S_k(\lambda(A))$, where $\lambda(A)=(\lambda_1,\ldots,\lambda_n)$ are the eigenvalues of the matrix $A$. Sometimes it is useful to express the derivatives of $S_k$ explicitly (see \cite[Prop. 2.2]{GM03}):
\eq{
S_k^{ij}(A) &:=\frac{\partial S_k(A)}{\partial A_{ij}}=\frac{1}{(k-1)!}\sum_{\substack{i_1,\ldots,i_{k-1}=1\\j_1,\ldots,j_{k-1}=1}}^{n} \delta_{i i_1\ldots i_{k-1}}^{j j_1 \ldots j_{k-1}}A_{i_1 j_1} \ldots A_{i_{k-1} j_{k-1}},\\
S_k^{ij,rs}(A) &:=\frac{\partial^2 S_k(A)}{\partial A_{ij}\partial A_{rs}}= \frac{1}{(k-2)!}\sum_{\substack{i_1,\ldots,i_{k-2}=1\\j_1,\ldots,j_{k-2}=1}}^{n} \delta_{i r i_1\ldots i_{k-2}}^{j s j_1 \ldots j_{k-2}}A_{i_1j_1}\ldots A_{i_{k-2}j_{k-2}},
}
When $A$ is diagonal, $S_k^{ij}(A)$ is diagonal with
\eq{
S_k^{ij}(A)=\left\{  
		\begin{aligned}
            &S_{k-1}(A|i), \quad &{\rm if } \ i=j,\\
            &0,             \quad &{\rm otherwise},
            \end{aligned}\right.
}
and
\eq{
		S_{k}^{ij, rs}(A) = \left\{  
		\begin{aligned}
			&S_{k-2} (A| ir), \quad& {\rm if } \ i=j, r=s, i \neq r,\\
			&-S_{k-2} (A|ij), \quad& {\rm if} \ i=s, j=r, i \neq j, \\
			&0, \quad& {\rm otherwise}.
		\end{aligned}
		\right.
		\label{Sm-ijrs}
}
Here $S_{k-1}(A|i)$ is the symmetric function with $\lambda_i=0$, and $S_{k-2}(A|ij)$ is the symmetric function with $\lambda_i=\lambda_j=0$. 

Let $F$ denote a symmetric, smooth functions of the eigenvalues of  $(A_{i}{}^j)=(\sigma^{jk}A_{ij})$ or as depending on $A=(A_{ij})$ and the spherical metric $\sigma=(\sigma_{ij})$, 
\eq{
F=F(\lambda_i)=F(A_i{}^j)=F(\sigma,A)
}
Denote by
\eq{
F_i^j :=\frac{\partial F}{\partial A_j{}^i},\quad F^{ij} := \frac{\partial F}{\partial A_{ij}}, \quad F^{ij,kl} := \frac{\partial^2 F}{\partial A_{ij}\partial A_{kl}}.
}
We collect some basic formulas for the elementary symmetric polynomial $S_k$ $(1 \leq k\leq n)$.
\begin{lemma}\cite{Gua13}\label{lem-Newton poly}
    For any integer $1 \leq k \leq n$, we have
		\begin{align}
		\sum_{i,j}(S_k)_i^j(A) \delta_j{}^i =&~ (n-k+1) S_{k-1}(A), \label{Sm-ij delta-ij}\\
		\sum_{i,j} (S_k)_i^j(A) A_j{}^i =&~ k S_k(A), \label{Sm-ij A-ij}\\
		\sum_{i,j} (S_k)_i^j(A) A_j{}^k A_k{}^i =&~ S_1(A) S_k (A) -(k+1) S_{k+1} (A). \label{Sm-ij A^2-ij}
				\end{align}
	\end{lemma}
Define the G{\aa}rding cone by
\begin{equation*}
		\Gamma_m := \{x \in \mathbb{R}^n:~S_i(x)>0, \ 1 \leq i \leq m  \}, \quad m=1, \ldots, n.
\end{equation*}
Equivalently, $\Gamma_m$ is the connected component of $\{x:~S_m(x)>0\} \subset \mathbb{R}^n$ which contains the positive cone 
$\Gamma_n=\{ x \in \bbR^n:~ x_i>0, \ 1 \leq i \leq n \}.
$
The normalized $k$-th elementary symmetric polynomial $p_k$ is defined by $p_k(\lambda)=\binom{n}{k}^{-1} S_k(\lambda)$ for all $1\leq k\leq n$, and $p_0=1$ and $p_k=0$ if either $i<0$ or $i>n$ by convention.

We assume that $F$ is a strictly increasing, homogeneous of degree one, and concave function of $\lambda(A)$ in some open cone $\Gamma\subset \bbR^n$ containing the positive cone $\Gamma_n$, and it is normalized that $F(1,\ldots, 1) =1$.
Then we have
\eq{
F^{ij} >0, \quad F^{ij} A_{ij} =F, \quad {F}^{ij, kl} \leq 0,
}
and
\eq{
\sum F^{ii} \geq 1.
}
It is well known that if the above $F$ is chosen as $p_{n-k}^{1/(n-k)}$ and $(p_n/p_k)^{1/(n-k)}$, then the corresponding open cones $\Gamma\subset\bbR^n$ are $\Gamma_{n-k}$ and $\Gamma_n$, respectively. We refer to \cite[Chap. 2]{Ger06b} for a detailed account and common properties of $F$.

\section{A priori estimates}
Throughout this section, we consider the positive, even solution $\varphi\in C^2(\bbS^n) $ to the following equation:  
\begin{equation}\label{eq-Lp-Weingarten-eq-hyperbolic by q}
	F(A[\varphi])= \g^q f^{\frac{1}{n-k}}, \quad q:= \frac{n+p}{n-k} -1,
\end{equation}
where $F=p_{n-k}^{1/(n-k)}$ or $F=(p_n/p_k)^{1/(n-k)}$, and $f$ is a smooth positive function on $\bbS^n$ that satisfies Assumption \ref{Assump-barrier-f} when $q \geq 1$ (i.e. $p \geq n-2k$). Moreover, if $F=p_{n-k}^{1/(n-k)}$, we assume that $A[\varphi]\in \Gamma_{n-k}\cap \overline{\Gamma}_n$; 
With the help of the full rank theorem (Theorem \ref{thm-full-rank}), this equation \eqref{eq-Lp-Weingarten-eq-hyperbolic by q} corresponds to the horospherical $p$-Christoffel-Minkowski problem \eqref{s1:horo-p-CM-problem}. If $F=(p_n/p_k)^{1/(n-k)}$, we assume that $A[\varphi]\in \Gamma_n$; This equation \eqref{eq-Lp-Weingarten-eq-hyperbolic by q} corresponds to the prescribed $p$-shifted Weingarten curvature problem \eqref{Lp-Weingarten-eq-hyperbolic}.

The horospherical support function $u(z)$ of an h-convex domain $\Omega$ is defined by
	\eq{
		u(z) = \inf \{ s \in \bbR:~ \Omega \subset \overline{B}_z(s) \}, \quad z \in \bbS^n,
	}
	which is exactly \eqref{def-u-signed-distance} when $\partial \Omega$  is smooth and strictly h-convex.
It was proved in \cite[Prop. 2.1]{LX22} that $\log \varphi$ is the horospherical support function of a non-empty, h-convex domain $\Omega :=\cap_{z \in \mathbb{S}^n} \overline{B}_z \( \log \varphi(z) \)$ when $A[\varphi]$ is positive semi-definite on $\mathbb{S}^n$. And it was proved in \cite[Cor. 2.3]{LX22} that the map $X:\mathbb{S}^n\ra \partial \Omega$ given by \eqref{X(z)} is surjective. If  $\varphi$ is in addition even on $\mathbb{S}^n$, i.e., $\varphi(z)= \varphi(-z)$ for all $z \in \mathbb{S}^n$, then we have that $\varphi \geq 1$ on $\mathbb{S}^n$ as $N=(0,1) \in \Omega$. Moreover, if $\varphi(z_0) =1$ for some $z_0 \in \mathbb{S}^n$ then  $\varphi(z) \equiv 1$, $A[\varphi(z)] \equiv 0$, and $ \cap_{z \in \mathbb{S}^n} \overline{B}_z ( \log \varphi(z) )$ is exactly the point $(0,1) \in \mathbb{H}^{n+1}$. Hence we can assume that $\varphi(z) >1$ in the a priori estimate. By using $X(z_1) \in \overline{B}_{z} \(\log \varphi(z) \)$, $\forall z \in \bbS^n$, where $z_1$ is a maximum point of $\varphi$ on $\bbS^n$, and the origin-symmetry of $\Omega$, it was proved in \cite[Lem. 7.2]{LX22} that
\begin{align}\label{eq-phi max < phi min}
	\cosh \( \log \g_{\max}  \) &\leq \g_{\min}, 
\end{align}
where 
\eq{
	\g_{\max} := \max_{ z \in \mathbb{S}^n} \g(z), \quad 	\g_{\min} := \min_{ z \in \mathbb{S}^n} \g(z).
 }
Then, as an application of the above \eqref{eq-phi max < phi min}, it was proved in \cite[Lem. 7.3]{LX22} that 
\eq{\label{eq-Dphi<phi}
|D \log \varphi(z)| <1, \quad \forall \ z \in \mathbb{S}^n.
}

Note that the Assumption \ref{Assump-barrier-f} on $f$ can be reformulated as follows: assume $0<f<2^{k-n}$ when $q=1$, and assume $0<f< (\frac{(q+1)^\frac{q+1}{2}}{(q-1)^\frac{q-1}{2}})^{k-n}$ when $q>1$.

\begin{lemma}[$C^0$-estimate]\label{lem-C0-est}
 We have
	\begin{equation}\label{eq-C0}
		1+\frac{1}{C} \leq \g \leq C,
	\end{equation}
	where $C>0$ is a constant depending only on $n$, $k$, $p$, $ f_{\min}$ and $\|f \|_{C^0}$.
\end{lemma}
\begin{proof}
	For simplicity, we define the function $\xi_q : (1, +\infty) \to \mathbb{R}_+$ by 
 \eq{ \label{def-zeta-q}
 \xi_q (t) =2 t^{q} (t-t^{-1})^{-1}.
 }
 Assume that $\g$ attains its maximum at $z_1 \in \mathbb{S}^n$. Then at $z_1$, we have $\varphi=\varphi_{\max}$, $D \g =0$ and $D^2 \g \leq 0$. Thus, at the point $z_1$,
	\eq{
		(f_{\min})^{-\frac{1}{n-k}} \geq \g_{\max}^{q} F^{-1} \geq 2\g_{\max}^{q}  \( \g_{\max}- \g_{\max}^{-1}\)^{-1} = \xi_q (\g_{\max}).
	}
        Similarly, at a minimum point $z_2\in \mathbb S^n$ of $\g$, we have
	\eq{
			(f_{\max})^{-\frac{1}{n-k}} \leq \g_{\min}^{q} F^{-1} \leq 2\g_{\min}^{q}  \( \g_{\min}- \g_{\min}^{-1}\)^{-1} = \xi_q (\g_{\min}).
	}
        \begin{enumerate}[(1)]
            \item If $q<1$, then $\xi_q (t)$ is strictly decreasing with $\xi_q (1^+) = +\infty $ and $\xi_q (+\infty) =0$. Thus, we have
        \eq{
		\g_{\max} \geq \xi_q^{-1} ( 	(f_{\min})^{-\frac{1}{n-k}} ), \quad \g_{\min} \leq \xi_q^{-1} ((f_{\max})^{-\frac{1}{n-k}}) .
	}
           \item If $q=1$, then $\xi_q(t)$ is strictly decreasing with $\xi_q (1^+) = +\infty $ and $\xi_q (+\infty) =2$. Since we assume $0<f<2^{k-n}$ in this case, it holds that
           \eq{
           \g_{\max} \geq \xi_q^{-1} ( (f_{\min})^{-\frac{1}{n-k}} ), \quad \g_{\min} \leq \xi_q^{-1} ((f_{\max})^{-\frac{1}{n-k}}) .
           }
        \item If $q>1$, then $\xi_q(t)$ is strictly decreasing on $(1,\sqrt{\frac{q+1}{q-1}}]$, and it is strictly increasing on $[\sqrt{\frac{q+1}{q-1}},\infty)$. Moreover, $\xi_q(1^+)=\infty$, $\xi_q(\sqrt{\frac{q+1}{q-1}})=\frac{(q+1)^{\frac{q+1}{2}}}{(q-1)^{\frac{q-1}{2}}}$ and $\xi_q(+\infty)=\infty$. Denote by $\xi_{q,-}^{-1}$ and $\xi_{q,+}^{-1}$ the inverse function of $\xi_q$ on the interval $(1,\sqrt{\frac{q+1}{q-1}}]$ and $[\sqrt{\frac{q+1}{q-1}},\infty)$ respectively. Since we assume that $0< f<\(\frac{(q+1)^\frac{q+1}{2}}{(q-1)^\frac{q-1}{2}}\)^{k-n}$ in this case, it holds that
        \begin{itemize}
            \item If $\varphi_{\max}\leq \sqrt\frac{q+1}{q-1}$, then 
            \eq{
            \varphi_{\max}\geq \xi_{q,-}^{-1}((f_{\min})^{-\frac{1}{n-k}}),\quad \varphi_{\min}\leq \xi_{q,-}^{-1}((f_{\max})^{-\frac{1}{n-k}});
            }
            \item If $\varphi_{\min}\geq \sqrt\frac{q+1}{q-1}$, then 
            \eq{
            \sqrt\frac{q+1}{q-1} \leq \varphi_{\min}\leq \varphi_{\max}\leq \xi_{q,+}^{-1}((f_{\min})^{-\frac{1}{n-k}});
            }
            \item If $\varphi_{\min} \leq \sqrt\frac{q+1}{q-1}\leq \varphi_{\max}$, then 
            \eq{
            \varphi_{\min}\leq \xi_{q,-}^{-1}((f_{\max})^{-\frac{1}{n-k}}),\quad \sqrt\frac{q+1}{q-1}\leq\varphi_{\max}\leq \xi_{q,+}^{-1}((f_{\min})^{-\frac{1}{n-k}}).
            }
        \end{itemize}
        \end{enumerate}
     
     Then by \eqref{eq-phi max < phi min}, there exists a constant $C>0$ such that 
	\eq{
		1+\frac{1}{C} \leq \g \leq C.
	}
	We complete the proof of Lemma \ref{lem-C0-est}.
\end{proof}

Combining Lemma \ref{lem-C0-est} with the estimate \eqref{eq-Dphi<phi}, we obtain the following a priori $C^1$-estimates.
%先验估计这里estimate是单数还是复数？好像都有。。。
\begin{lemma}[$C^1$-estimate]\label{lem-C1-est}
	We have 
	\begin{equation}\label{eq-C1}
		|D \g| \leq C,
	\end{equation}
	where $C>0$ is a constant depending only on $n$, $k$, $p$, $f_{\min}$ and $\|f \|_{C^0 }$.
\end{lemma}

\begin{lemma}[$C^2$-estimate]\label{lem-phi C2}
	We have
	\begin{equation}\label{eq-est C2}
	   \|\g \|_{C^2} \leq C,
	\end{equation}
	where $C>0$ is a constant depending only on $n$, $k$, $p$, $f_{\min}$ and $ \|f \|_{C^2}$.
\end{lemma}

\begin{proof}
	It follows from $A[\g] \geq 0$, the $C^0$ and $C^1$ estimates of $\g$ in \eqref{eq-C0} and \eqref{eq-C1} that $D^2 \g \geq -C$. Note again that $A[\g]$ is positive semi-definite. Then 
	\eq{
		D^2 \g \leq A[\g] +C \leq \tr A[\g] +C \leq \Delta \g +C.
	}
	Consequently, we have
	\begin{equation}\label{eq-D^2 phi}
		- C \leq D^2 \g \leq \Delta \g+C.
	\end{equation}
	    Hence, in order to prove $ \|\g \|_{C^2} \leq C$, it suffices to show that $\Delta \g \leq C$.  
	
	Assume that $\Delta \g$ attains its maximum at $z_0 \in \mathbb{S}^n$. Without loss of generality, we assume that $\Delta \g (z_0) \geq 1$. Take a normal coordinate system on $\mathbb{S}^n$ around $z_0$ such that $\sigma_{ij} = \delta_{ij}$ and that $D^2 \g$ is diagonal at $z_0$. Consequently, $F^{ij}$ and $A_{ij} [\g]$ are both diagonal at $z_0$. Then, at $z_0$ we have
	\eq{
		(\Delta \g)_i = 0, \quad (\Delta \g)_{ij} \leq 0.
	}
	Since $F$ is increasing and homogeneous of degree one, using the Ricci identity, at $z_0$ we have
	\eq{\label{eq-0 geq F-ii Delta phi-ii} 
	0 \geq &~F^{ii} (\Delta \g)_{ii}  \\
	= &~ F^{ii} \( \Delta (\g_{ii}) +2 \Delta \g - 2n \g_{ii} \)  \\
	= &~ F^{ii} \Delta (A_{ii} [\g]) + \sum F^{ii} \Delta \( \frac{1}{2} \frac{|D \g|^2}{\g} - \frac{1}{2} \( \g-\frac{1}{\g} \)  \)  \\
    &~\quad + 2\sum F^{ii} \Delta \g -2n F^{ii}  \g_{ii} \\
	=&~ F^{ii} \Delta (A_{ii} [\g]) + \sum F^{ii} \Delta \( \frac{1}{2} \frac{|D \g|^2}{\g} - \frac{1}{2} \( \g-\frac{1}{\g} \)  \)+2\sum F^{ii} \Delta \g  \\
	&~ -2n F^{ii} A_{ii} [\g] - 2n \sum F^{ii} \( \frac{1}{2} \frac{|D \g|^2}{\g} - \frac{1}{2} \(\g -\frac{1}{\g}\) \)  \\
	\geq&~ F^{ii} \Delta (A_{ii} [\g]) + \frac{1}{2}\sum F^{ii} \Delta \frac{|D \g|^2}{\g}-2n \g^q f^{\frac{1}{n-k}} -
	\sum F^{ii} \( C \Delta \g + C \),
	}	where $C$ is a constant depending only on $n$, $k$, $q$, $ \|f \|_{C^0}$ and $f_{\min}$.
	
	For the first term on the RHS of \eqref{eq-0 geq F-ii Delta phi-ii}, the concavity of $F$ implies
	\eq{
		\Delta F = F^{ij} \Delta A_{ij} [\g] +{F}^{ij, kl} \metric{D A_{ij} [\g]}{D A_{kl} [\g]} \leq F^{ii} \Delta A_{ii}[\g].
	}
	On the other hand, equation \eqref{eq-Lp-Weingarten-eq-hyperbolic by q} yields
	\eq{
		\Delta F = \Delta (\g^q f^{\frac{1}{n-k}}) \geq -C \Delta \g -C,
	}
	where $C$ is a positive constant depending only on $n$, $k$, $q$, $ \|f \|_{C^2}$. Thus, we obtain
	\begin{equation}\label{eq-est-F Delta-A phi}
		F^{ii} \Delta (A_{ii} [\g])  \geq -C \Delta \g- C.
	\end{equation}
	
	Now we estimate $\frac{1}{2} \Delta \frac{|D \g|^2}{\g}$. A direct calculation at $z_0$ yields
	\eq{\label{eq-est-Delta-Dphi}
		\frac{1}{2} \Delta \frac{|D \g|^2}{\g}= &~\frac{|D^2 \g|^2}{\g} + \frac{\g_i (\Delta \g)_i}{\g} + (n-1) \frac{|D\g|^2}{\g} \\
      &~- 2 \frac{\g_{ij} \g_i \g_j}{\g^2} -\frac{1}{2}\frac{|D \g|^2}{\g^2} \Delta \g  +\frac{|D \g|^4}{\g^3}, \\
		\geq&~ \frac{1}{C} (\Delta \g)^2 -C \Delta \g - C,  
	}
	where we used $|D^2\g|^2 \geq (\Delta \g)^2/n$, $(\Delta \g)_i (z_0) =0$ and \eqref{eq-D^2 phi}. Substituting \eqref{eq-est-F Delta-A phi} and \eqref{eq-est-Delta-Dphi} into \eqref{eq-0 geq F-ii Delta phi-ii}, and noting that $\sum F^{ii} \geq 1$, we  have
	\eq{
		0 \geq \frac{1}{C} (\Delta \g)^2 - C \Delta \g -C.
	}
	Thus we have $\Delta \g \leq C$ and then $ \|\g \|_{C^2} \leq C$. We complete the proof of Lemma \ref{lem-phi C2}.
\end{proof}

\begin{lemma}\label{lem-elliptic-(p_n/p_{n-k})^{1/k}}
	Let $F=(p_n/p_k)^{1/(n-k)}$. Then
	\begin{equation} \label{eq-est A phi >1/c}
			A[\g] \geq C,
	\end{equation}
	where $C>0$ is a constant depending only on $n$, $k$, $q$, $f_{\min}$ and $ \|f \|_{C^2}$. Therefore, the associated operator $F^{kl}D_kD_l$ is uniformly elliptic at $\varphi$. 
\end{lemma}
\begin{proof}
    It follows from the Newton-MacLaurin inequality and equation \eqref{eq-Lp-Weingarten-eq-hyperbolic by q} that
	\eq{
	  p_n^{\frac{1}{n}} (A [\g]) \geq \(\frac{p_n(A[\g]) }{p_k(A[\g])}\)^{\frac{1}{n-k}}  =\g^q f^{\frac{1}{n-k}} \geq \frac{1}{C}.
	}
	On the other hand, the uniform upper bound of $A [\g]$ follows from \eqref{eq-est C2}. So we have $A [\g] \geq C$. Thus, the operator $F^{kl}D_kD_l$ is uniformly elliptic at $\varphi$.
\end{proof}

\begin{lemma}\label{lem-elliptic-(p_k)^{1/k}}
    Let $F=p_{n-k}^{1/(n-k)}$. Then the associated operator $F^{kl}D_kD_l$ is uniformly elliptic at $\varphi$. 
\end{lemma}
\begin{proof}
     Note that $F(A[\g])=p_{n-k}^{1/(n-k)}(A[\g])=\g^q f^{\frac{1}{n-k} }\geq C>0$ follows from the $C^0$-estimate \eqref{eq-C0} of $\g$ and $f_{\min}$. %$\|f\|_{C^0}$. 
     This together with the upper bound of $A[\varphi]$, implies that the eigenvalues of $A[\varphi]$ lie in a compact subset of the cone $\Gamma_k$. Thus, the operator $F^{kl}D_kD_l$ is uniformly elliptic at $\varphi$. 
\end{proof}

By Lemma \ref{lem-elliptic-(p_n/p_{n-k})^{1/k}} and Lemma \ref{lem-elliptic-(p_k)^{1/k}} and the a priori $C^2$ estimates in \eqref{eq-est C2}, the higher order derivative estimates of $\g$ follows from the H\"older estimate of Krylov-Evans \cite{Kry82} and the Schauder theory \cite{GT01}.
\begin{theorem}\label{thm-regularity-estimate}
	Let $F=(p_n/p_k)^{1/(n-k)}$ or $F=p_{n-k}^{1/(n-k)}$. For any integer $l \geq 2$, there exists a constant $C>0$ depending only on $n$, $k$, $p$, $f_{\min}$ and $ \|f \|_{C^{l, \alpha}}$ such that 
	\begin{equation}\label{eq-higher order derivative}
		 \|\g \|_{C^{l+2, \alpha}} \leq C.
	\end{equation}
\end{theorem}

\section{A deformation lemma}
In this section, we establish the key deformation lemma (Lemma \ref{deformation-lemma}), which will be used to prove the full rank theorem (Theorem \ref{thm-full-rank}) to ensure that the strict h-convexity is preserved along a homotopic path when we prove the mains theorems using the degree method.

\begin{lemma}[Deformation lemma]\label{deformation-lemma}
Let $O \subset \mathbb{S}^n$ be an open subset. Suppose $\varphi \in C^4(O)$ is a solution of 
\eq{\label{eq-in deformation lemma}
S_k(A[\varphi (z)]) = \varphi^{n+p-k} f(z)
}
in $O$, and that the matrix $A[\varphi]$ is positive semi-definite. Suppose there is a positive constant $C_0>0$ such that for a fixed integer $k\leq \ell \leq n-1$, $S_{\ell} (A[\varphi(z)]) \geq C_0$ for all $z \in O$. Let $\phi(z) = S_{\ell+1} (A[\varphi(z)] )$ and let $\tau(z)$ be the largest eigenvalue of the matrix
\eq{\label{deformation-lemma-cond on f}
 &-D^2(f^{-\frac{1}{k} } ) -\frac{2(k-p-n)}{k} d(f^{-\frac{1}{k} }) d\log \varphi \\
 &+ \metric{D f^{-\frac{1}{k} }}{D \log \varphi} \sigma+  \frac{(k-p-n)(n+p)}{k^2}f^{ -\frac{1}{k} } (d \log \varphi)^2 \\
 &- \(\frac{n+p}{2k} \frac{|D \varphi|^2}{\varphi^2}  + \frac{n+p}{2k} + \frac{2k-p-n}{2k} \frac{1}{\varphi^2} \)f^{ -\frac{1}{k} }  \sigma.
} 
Then there are constants $C_1$, $C_2$ depending only on $||\varphi||_{C^3}$, $||f||_{C^2}$, $n$, $k$, $p$ and $C_0$, such that the differential inequality
\eq{
\sum_{\a, \b=1}^n S_k^{\a\b} \phi_{\a\b}(z) \leq  k (n-\ell)\varphi^{n+p-k} f^{\frac{k+1}{k} } S_{\ell} (A[\varphi]) \tau(z) + C_1 |D \phi (z)| +C_2 \phi (z)
}
holds in $O$.
\end{lemma}

Following the notations of \cite{CF85} and \cite{GM03}, for two functions $\cA(z)$, $\cB(z)$ defined in an open set $O\subset \mathbb{S}^n$, $z \in O$, we say that $\mathcal{A}(z)  \lesssim \mathcal{B}(z)$ if there exist positive constants $c_1$ and $c_2$ depending only on $||\varphi||_{C^3}$, $||f||_{C^2}$, $n$, $k$, $p$ and $C_0$ such that
	\eq{
		(\mathcal{A}-\mathcal{B})(z) \leq (c_1 |D \phi| +c_2 \phi) (z),
	}
where $\phi (z) = S_{\ell+1} (A[\g(z)])$. 
We also write $\mathcal{A}(z) \sim \mathcal{B}(z)$ if $\mathcal{A}(z) \lesssim \mathcal{B}(z)$ and $\mathcal{B}(z) \lesssim \mathcal{A}(z)$.
	For any $z\in O$ fixed, we choose a local orthonormal frame $\{e_1,\ldots,e_n\}$ so that the matrix $A[\varphi]$ is diagonal at $z$, and let $\lambda_1 \geq \cdots \geq \lambda_n$ be its eigenvalues. Since $S_\ell (A[\varphi]) \geq C_0$ and $A[\varphi] \geq 0$, there exists a positive constant $C$ depending on $\|\varphi\|_{C^3}$, $\|f\|_{C^2}$, $n$ and $C_0$ such that
	\eq{\label{assump on lambda}
		\lambda_1 \geq \cdots \geq \lambda_{\ell} \geq C \geq \lambda_{\ell+1} \geq \cdots \geq \lambda_n \geq 0.
	}	
	Let $G = \{1,2,\ldots, \ell\}$ and $B= \{\ell+1, \ldots, n\}$ be the ``good" and ``bad" sets of indices respectively. It follows from \eqref{assump on lambda} that
	\eq{
		0\sim \phi(z) = S_{\ell+1} (\lambda)=
  \sum_{1 \leq i_1<\cdots<i_{\ell+1} \leq n } \lambda_{i_1} \cdots \lambda_{i_{\ell+1}} \sim S_{\ell}(G) \bigg( \sum_{i \in B } \lambda_i \bigg) \sim \bigg( \sum_{i \in B} \lambda_i \bigg).
	}
	So we have 
	\eq{\label{lambda-0 B}
		\lambda_i \sim 0, \quad \forall \ i \in B.
	}
	This yields that
	\begin{itemize}
		\item 	For any $ m \geq \ell+1$, we have
		\eq{\label{Sm=0,m>l}
			S_{m} (A) \sim 0.
		}
		\item 	For any $1 \leq m  \leq  \ell$, we have
		\eq{\label{Sm(G)}
			S_m (A) \sim S_m (G), \quad S_m(A|j) \sim 
			\left\{  
			\begin{aligned}
				&S_m (G|j) , \quad &j \in G; \\
				&S_m (G), \quad &j \in B.
			\end{aligned}
			\right.	
		}
            and
            \eq{
            S_m(A|ij)\sim \left\{  
			\begin{aligned}
				&S_m (G|ij) , \quad &i,j \in G; \\
				&S_m (G|j),   \quad &i\in B,j\in G;\\
                    &S_m (G), \quad &i,j\in B,i\neq j.
			\end{aligned}
			\right.	\label{Sm(A|ij)}
            }
            We also set $S_{-1}=0$ and $S_{0}=1$ by convention.
		\item For all $i,j \in G$ with $i \neq j$,
		\eq{\label{S_l(A|i)=S_(l-1) (A|ij)=0}
			S_{\ell} (A| i) \sim 0, \quad S_{\ell-1} (A| ij) \sim 0. 
		}
	\end{itemize}
	Then by \eqref{Sm(G)} and \eqref{S_l(A|i)=S_(l-1) (A|ij)=0}, we have
	\eq{
		0 \sim \phi_{,\alpha} = S_{\ell+1}^{ij} A_{ij, \alpha} = \sum_i S_{\ell}(A| i) A_{ii, \a} \sim S_{\ell} (G)\sum_{i\in B} A_{ii, \alpha}.
	}
	In view of $C \geq S_\ell (A) \geq C_0$
 %  \eqref{s5:S-ell-condition} 
    and $S_{\ell}(A)\sim S_{\ell}(G)$, we get
	\begin{equation}\label{Aii,a-0}
		\sum_{i \in B} A_{ii, \a} \sim 0, \quad \forall \ \a,
	\end{equation}
	and thus,
	\eq{\label{A_iia*A_jja}
	\sum_{\substack{i,j \in B\\ i \neq j}} A_{ii, \a} A_{jj, \a} = \(\sum_{i \in B} A_{ii, \a} \)^2- \sum_{i \in B} A_{ii, \a}^2 \sim  -\sum_{i \in B} A_{ii, \a}^2.	
	}
	For convenience, we denote $F=S_k(A)$ in this section.
	\begin{lemma}\label{lem-4.3}
    It holds that
		\eq{\label{Faa phiaa-v1}
			\sum_{\a}F^{\a\a} \phi_{\a\a} 
			\sim&  \sum_{i,\a}F^{\a\a} S_\ell (A|i) A_{ii, \a\a} \\
                    &- S_{\ell-1}(G)\bigg(\sum_{\substack{\a \in G\\ i,j \in B}} S_{k-1}(G|\a) A_{ij, \a}^2
			+\sum_{\a, i, j \in B} S_{k-1}(G) A_{ij,\a}^2\bigg)   \\
			&- 2 \bigg(\sum_{\substack{\a, i\in G\\ j \in B}}  S_{\ell-1} (G|i)S_{k-1}(G| \a) A_{ij, \a}^2+ \sum_{\substack{i \in G\\\a, j \in B}}  S_{\ell-1} (G|i)S_{k-1}(G) A_{ij, \a}^2 \bigg).	
			} 
	\end{lemma}	
	
	\begin{proof}
		Differentiating $\phi(z) = S_{\ell+1} (A[\varphi] )$ twice gives
		\eq{
			\phi_{\a\a} &= S_{\ell+1}^{ij} A_{ij, \a\a} +S_{\ell+1}^{ij, rs} A_{ij, \a} A_{rs, \a}\\
			&=\sum_{i} S_{\ell} (A| i) A_{ii, \a\a} +  S_{\ell+1}^{ij, rs} A_{ij, \a} A_{rs, \a} .
		}
		Using \eqref{Sm-ijrs}, we have
		\eq{ \label{2nd-term}
			S_{\ell+1}^{ij, rs} A_{ij, \a} A_{rs, \a}=&~ \sum_{i \neq r} S_{\ell-1} (A| ir) A_{ii, \a} A_{rr, \a} - \sum_{i \neq j} S_{\ell-1} (A|ij) A_{ij, \a}^2 \\
                =&~\bigg(2\sum_{\substack{i\in B\\ j\in G}}+\sum_{\substack{i,j\in B\\i\neq j}}+\sum_{\substack{i,j\in G\\i\neq j}}\bigg) S_{\ell-1}(A|ij)(A_{ii,\a}A_{jj,\a}-A_{ij,\a}^2).
            }
            It follows from \eqref{Sm(A|ij)}, \eqref{S_l(A|i)=S_(l-1) (A|ij)=0}, \eqref{Aii,a-0} and \eqref{A_iia*A_jja} that
            \eq{ \label{mid-step-1}
            \sum_{\substack{i\in B\\j\in G}}S_{\ell-1}(A|ij)A_{ii,\a}A_{jj,\a} &\sim \( \sum_{j\in G} S_{\ell-1}(G|j) A_{jj,\a}\) \sum_{i\in B}A_{ii,\a} \sim 0,
            }
            \eq{ \label{mid-step-2}
            -\sum_{\substack{i\in B\\j\in G}}S_{\ell-1}(A|ij)A_{ij,\a}^2 &\sim -\sum_{\substack{i\in B\\j\in G}} S_{\ell-1}(G|j) A_{ij,\a}^2,
            }
            \eq{ \label{mid-step-3}
            \sum_{\substack{i,j\in G\\i\neq j}}S_{\ell-1}(A|ij)(A_{ii,\a}A_{jj,\a}-A_{ij,\a}^2)
            \sim 0,
            }
            \eq{ \label{mid-step-4}
            \sum_{\substack{i,j\in B\\i\neq j}}S_{\ell-1}(A|ij)(A_{ii,\a}A_{jj,\a}-A_{ij,\a}^2) &\sim S_{\ell-1}(G)\sum_{\substack{i,j\in B\\i\neq j}}(A_{ii,\a}A_{jj,\a}-A_{ij,\a}^2)\\
            &\sim -S_{\ell-1}(G)\bigg(\sum_{i\in B} A_{ii,\a}^2+\sum_{\substack{i,j\in B\\i\neq j}}A_{ij,\a}^2\bigg)\\
            &\sim -S_{\ell-1}(G)\sum_{i,j\in B}A_{ij,\a}^2.
            }
            Substituting \eqref{mid-step-1}, \eqref{mid-step-2}, \eqref{mid-step-3} and \eqref{mid-step-4} into \eqref{2nd-term}, we obtain
            \eq{
            S_{\ell+1}^{ij,rs}A_{ij,\a}A_{rs,\a} &\sim -2 \sum_{\substack{i \in B\\j\in G}} S_{\ell-1} (G|j) A_{ij,\a}^2 - \sum_{i,j \in B} S_{\ell-1}(G) A_{ij, \a}^2.
		}
		Since $F= S_k(A)$, it follows from \eqref{Sm(G)} that
		\eq{
			F^{\a\a} \sim  \left\{
			\begin{aligned}
				&S_{k-1} (G| \a ), \quad &\a \in G, \\
				&S_{k-1} (G) , \quad &\a \in B.
			\end{aligned}
			\right.
		}
		Contracting $\phi_{\a\a}$ with $F^{\a\a}$, we then have
		\eq{\label{Faa-phiaa-1}
                 &~\sum_{\a} F^{\a\a} \phi_{\a\a} \\
			=&~ \sum_{\a}F^{\a\a}  ( S_{\ell+1}^{ii} A_{ii,\a\a}  +S_{\ell+1}^{ij, rs} A_{ij, \a} A_{rs, \a})  \\
			\sim &~ \sum_{\a,i} F^{\a\a} S_\ell (A|i) A_{ii, \a\a} - S_{\ell-1}(G) \sum_{\a}\sum_{i,j \in B}F^{\a\a} A_{ij,\a}^2 - 2 \sum_{\a}\sum_{\substack{i \in G\\ j \in B}} S_{\ell-1} (G|i) F^{\a\a} A_{ij, \a}^2   \\
			\sim &~ \sum_{\a,i}F^{\a\a} S_\ell (A|i) A_{ii, \a\a} - S_{\ell-1}(G)\bigg(\sum_{\substack{\a \in G\\ i,j \in B}} S_{k-1}(G|\a) A_{ij, \a}^2
			+\sum_{\a, i, j \in B} S_{k-1}(G) A_{ij,\a}^2 \bigg)   \\
			&~- 2 \bigg(\sum_{\substack{\a, i\in G\\ j \in B}}  S_{\ell-1} (G|i)S_{k-1}(G| \a) A_{ij, \a}^2+ \sum_{\substack{i \in G\\\a, j \in B}}  S_{\ell-1} (G|i)S_{k-1}(G) A_{ij, \a}^2 \bigg). 
		}
		This completes the proof of Lemma \ref{lem-4.3}.
	\end{proof}

        By commuting the covariant derivatives and using some basic properties of the elementary symmetric polynomials, we convert \eqref{Faa phiaa-v1} into the following useful form.
	\begin{lemma}\label{lemma-Faa-phiaa}
        Denote $F(A)=S_k(A[\varphi] )=h(z)$. 
		Then
		\eq{\label{Faa-phiaa=I1+...+I5}
		\sum_{\a} F^{\a\a} \phi_{\a\a} \sim I_1 +I_2+I_3+I_4+I_5,
		}
		where
		\begin{align}
			I_1:=&~  S_{\ell}(G) \sum_{i \in B} \( h_{ii} -\metric{\frac{D \g}{\g}}{Dh}-\frac{\psi}{\varphi} kh - \frac{1}{\varphi}(S_1(G) h- (k+1) S_{k+1} (G))  \)\\
            &+\frac{2}{\varphi^2}\sum_{\substack{\a \in G\\ i \in B}} S_{\ell}(G)    S_{k-1} (G| \a) A_{\a\a} \varphi_{\a}^2, \label{def-I1}\\
			I_2:=&~ -\sum_{\substack{i \in B\\ \a, \b \in G\\ \a \neq \b}}S_\ell(G)
			 S_{k-2} (G|\a\b) A_{\a\a, i} A_{\b\b, i}-  \sum_{\substack{\a \in G\\ i,\b \in B}}S_{\ell-1}(G) S_{k-1}(G|\a) A_{i\b, \a}^2 \nonumber\\
			&-\sum_{\substack{\a, \b \in G\\ i \in B}} S_{\ell-1} (G|\b) S_{k-1}(G| \a) A_{i\b, \a}^2, \label{def-I2}\\
			I_3:=&~ \sum_{i, \a,\b \in B}S_\ell(G)S_{k-2} (G)   A_{\a\b, i}^2 -\sum_{\a, i, \b \in B} S_{\ell-1}(G) S_{k-1}(G) A_{i\b,\a}^2, \label{def-I3}\\
			I_4:=&~  \sum_{\substack{i \in B\\ \a,\b \in G\\ \a \neq \b}} S_\ell (G) S_{k-2} (G| \a\b) A_{\a\b, i}^2
			-\sum_{\substack{\a, \b \in G\\ i \in B}} S_{\ell-1} (G|\b) S_{k-1}(G|\a) A_{i\b, \a}^2, \label{def-I4}\\
			I_5:=&~ 2 \sum_{\substack{\a \in G\\ i, \b \in B}} S_\ell(G)S_{k-2} (G| \a) A_{\a\b, i}^2- 2 \sum_{\substack{\a \in G\\ i, \b \in B}} S_{\ell-1}(G|\a) S_{k-1}(G) A_{i\a, \b}^2.\label{def-I5}
		\end{align}	
	\end{lemma}

	\begin{proof} 
	Differentiating both sides of the equation
    $F= S_k (A[\varphi(z)])=h(z)$ and using \eqref{Aii,a-0}, we have
		\eq{
		h_i= \sum_{\a} F^{\a\a} A_{\a\a, i} =&~ \sum_{\a \in G} S_{k-1} (G|\a) A_{\a\a, i} + S_{k-1}(G)\sum_{\a\in B} A_{\a\a, i} \\
		\sim &~ \sum_{\a \in G} S_{k-1} (G|\a) A_{\a\a, i}.
	}
	Using \eqref{Sm-ijrs}, we get
	\eq{
		h_{ii} =&~ F^{\a\b} A_{\a\b, ii}+ F^{\a\b, \mu\gamma} A_{\a\b, i} A_{\mu\gamma, i}  \\
	=&~ \sum_{\a}F^{\a\a} A_{\a\a, ii} + \sum_{\a \neq \b} S_{k-2}(A| \a\b) A_{\a\a, i} A_{\b\b, i} -\sum_{\a \neq \b} S_{k-2} (A|\a\b) A_{\a\b, i}^2.
	}
	Then by \eqref{Sm(A|ij)}, 
 \eqref{Aii,a-0} and \eqref{A_iia*A_jja} we have
	\eq{\label{Faa-Aaaii-v1}
		&~\sum_{\a}F^{\a\a} A_{\a\a,ii} \\
	=&~h_{ii} -\sum_{\a \neq \b} S_{k-2}(A| \a\b) A_{\a\a, i} A_{\b\b, i} +\sum_{\a \neq \b} S_{k-2} (A|\a\b) A_{\a\b, i}^2\\
    \sim&~h_{ii} -\sum_{\substack{\a,\b\in G\\ \a \neq \b}} S_{k-2}(G| \a\b) A_{\a\a, i} A_{\b\b, i} -\sum_{\substack{\a,\b\in B\\ \a \neq \b}} S_{k-2}(G) A_{\a\a, i} A_{\b\b, i}\\
    &~+  \sum_{\a \neq \b} S_{k-2} (G| \a\b) A_{\a\b, i}^2\\
	\sim&~ h_{ii}- \sum_{\substack{\a,\b\in G\\ \a \neq \b}} S_{k-2}(G|\a\b) A_{\a\a, i} A_{\b\b, i}+ S_{k-2} (G)\sum_{\a\in B} A_{\a\a, i}^2\\
	&~+ \sum_{\a \neq \b} S_{k-2} (G|\a\b) A_{\a\b, i}^2. 
	}
    Now we begin to convert \eqref{Faa phiaa-v1}.
	For the first term on the RHS of \eqref{Faa phiaa-v1}, we use the commutative formula \eqref{s5:Aiiaa-Aaaii} to derive
	\eq{
	\sum_{i, \a} F^{\a\a} S_{\ell}(A|i) A_{ii, \a\a} = &~	\sum_{i, \a} F^{\a\a} S_{\ell}(A|i) \bigg( A_{\a\a, ii} + \sum_l \frac{\varphi_l}{\varphi} (A_{ii,l} - A_{\a\a, l}) +T_{i\a} \bigg),}
	where
	\eq{\label{def-Tia}
		T_{i\a} :=\frac{1}{\g} (A_{ii}^2 -A_{\a\a}^2) +\frac{\psi}{\g}(A_{ii} -A_{\a\a}) - \frac{2}{\g^2} ( A_{ii}\g_i^2 - A_{\a\a} \g_{\a}^2 ).
	}
	To proceed further, it follows from \eqref{Sm-ij delta-ij}, \eqref{S_l(A|i)=S_(l-1) (A|ij)=0} and \eqref{Aii,a-0} that
	\eq{ 
	   &~\sum_{i, \a} F^{\a\a} S_{\ell}(A|i) \bigg(\sum_l  \frac{\varphi_l}{\varphi} (A_{ii,l} - A_{\a\a, l}) \bigg)\\
	= &~\sum_{\a} F^{\a\a} \sum_{i,l} \frac{\g_l}{\g} S_\ell(A|i) A_{ii,l} - \sum_{i} S_\ell(A|i) \sum_{\a,l} \frac{\g_l}{\g}F^{\a\a} A_{\a\a, l} \\
	\sim &~  \sum_{\a} F^{\a\a} \sum_{l} \frac{\g_l}{\g} S_\ell(G) \bigg(\sum_{i \in B}A_{ii,l}\bigg) -   (n-\ell) S_{\ell} (G) \metric{\frac{D \g}{\g}}{Dh}\\
	\sim &~  -(n-\ell) S_{\ell} (G)  \metric{\frac{D \g}{\g}}{Dh}.
	}
     Then we get
	\eq{\label{Faa-S_l-Aiiaa}
	\sum_{i, \a} F^{\a\a} S_{\ell}(A|i) A_{ii, \a\a} \sim 
	\sum_{i, \a} F^{\a\a} S_{\ell}(A|i) \( A_{\a\a, ii} +T_{i\a} \)-(n-\ell) S_{\ell} (G)  \metric{\frac{D \g}{\g}}{Dh}.
	}
	Plugging \eqref{Faa-Aaaii-v1} into the above formula \eqref{Faa-S_l-Aiiaa} and using 
    \eqref{S_l(A|i)=S_(l-1) (A|ij)=0}, we obtain
	\eq{\label{Faa-S_l-A_iiaa}
		\sum_{i, \a} F^{\a\a} S_{\ell}(A|i) A_{ii, \a\a}
		\sim&~ \sum_i S_{\ell}(A|i) h_{ii} - \sum_i \sum_{\substack{\a, \b \in G\\ \a \neq \b}} S_{\ell} (A|i) S_{k-2} (G| \a\b) A_{\a\a,i} A_{\b\b,i} \\
           &~+ \sum_{i, \a} F^{\a\a} S_{\ell} (A|i) T_{i \a}+ S_{k-2} (G) \sum_i \sum_{\a \in B} S_{\ell} (A|i) A_{\a\a,i}^2\\
           &~+ \sum_i \sum_{\a \neq \b} S_{\ell} (A|i) S_{k-2} (G| \a\b) A_{\a\b,i}^2 
           -(n-\ell) S_{\ell}(G) \metric{\frac{D \varphi}{\varphi}}{Dh}\\
		\sim&~ S_{\ell} (G) \sum_{i \in B} h_{ii} -  \sum_{\substack{i \in B \\ \a, \b \in G \\ \a \neq \b}} S_{\ell} (G) S_{k-2} (G| \a\b) A_{\a\a,i} A_{\b\b,i} \\
            &~+ \sum_{i, \a} F^{\a\a} S_{\ell} (A|i) T_{i \a}+ S_{\ell}(G) S_{k-2} (G) \sum_{i, \a \in B} A_{\a\a,i}^2\\
		&~+ \sum_{i\in B} \bigg( \sum_{\substack{\a, \b \in G\\ \a \neq \b}}+ 2\sum_{\substack{\a \in G\\ \b \in B}} + \sum_{\substack{\a, \b \in B \\ \a \neq \b}}\bigg) S_{\ell} (G) S_{k-2} (G| \a\b) A_{\a\b,i}^2\\
   &~-(n-\ell) S_{\ell}(G) \metric{\frac{D \varphi}{\varphi}}{Dh}.
	}
	Inserting \eqref{Faa-S_l-A_iiaa} into \eqref{Faa phiaa-v1}, we get
	\eq{
		\sum_{\a}F^{\a\a} \phi_{\a\a} 
        \sim&~ S_{\ell} (G) \sum_{i \in B} h_{ii} + \sum_{i, \a} F^{\a\a} S_{\ell} (A|i) T_{i\a}-(n-\ell) S_{\ell} (G)  \metric{\frac{D \g}{\g}}{Dh}\\
            &~+ S_\ell(G)S_{k-2} (G)  \sum_{i, \a \in B} A_{\a\a, i}^2+ \sum_{\substack{i \in B \\ \a,\b \in G \\ \a \neq \b}} S_\ell (G) S_{k-2} (G| \a\b) A_{\a\b, i}^2 \\
		&~
		+2 \sum_{\substack{\a \in G\\ i, \b \in B}} S_\ell(G)S_{k-2} (G| \a) A_{\a\b, i}^2 + \sum_{\substack{i, \a, \b \in B\\ \a \neq \b}} S_\ell(G) S_{k-2} (G) A_{\a\b, i}^2\\
		&~-\sum_{\substack{i \in B\\ \a, \b \in G\\ \a \neq \b}}
		S_\ell(G) S_{k-2} (G|\a\b) A_{\a\a, i} A_{\b\b, i}- \sum_{\substack{\a \in G\\ i,j \in B}} S_{\ell-1}(G)S_{k-1}(G|\a) A_{ij, \a}^2\\
		&~
		-\sum_{\a, i, j \in B} S_{\ell-1}(G) S_{k-1}(G) A_{ij,\a}^2-2 \sum_{\substack{\a, i\in G\\ j \in B}} S_{\ell-1} (G|i) S_{k-1}(G| \a) A_{ij, \a}^2\\
		&~- 2 \sum_{\substack{i \in G\\ \a, j \in B}} S_{\ell-1}(G|i) S_{k-1}(G) A_{ij, \a}^2.
  } 
         By changing the indices and rearranging the terms properly, we finally obtain
         \eq{
             \sum_{\a}F^{\a\a} \phi_{\a\a}\sim&~ S_{\ell} (G) \sum_{i \in B} h_{ii} + \sum_{i, \a} F^{\a\a} S_{\ell} (A|i) T_{i\a}-(n-\ell) S_{\ell} (G)  \metric{\frac{D \g}{\g}}{Dh}\\
            &~-\underbrace{S_\ell(G)\sum_{\substack{i \in B\\ \a, \b \in G\\ \a \neq \b}}
		 S_{k-2} (G|\a\b) A_{\a\a, i} A_{\b\b, i}- S_{\ell-1}(G)\sum_{\substack{\a \in G \\ i,\b \in B}} S_{k-1}(G|\a) A_{i\b, \a}^2}_{I_2}\\
            &~-\underbrace{\sum_{\substack{\a, \b \in G\\ i \in B}} S_{\ell-1} (G|\b) S_{k-1}(G| \a) A_{i\b, \a}^2}_{I_2}\\
            &~+ \underbrace{S_\ell(G)S_{k-2} (G)\sum_{i, \a,\b \in B} A_{\a\b, i}^2-S_{\ell-1}(G) S_{k-1}(G)\sum_{\a, i, \b \in B} A_{i\b,\a}^2}_{I_3}\\
            &+ \underbrace{S_\ell (G)\sum_{\substack{i \in B\\ \a,\b \in G\\ \a \neq \b}}  S_{k-2} (G| \a\b) A_{\a\b, i}^2- \sum_{\substack{\a, \b \in G\\ i \in B}} S_{\ell-1} (G|\b) S_{k-1}(G| \a) A_{i\b, \a}^2}_{I_4}\\
		&+\underbrace{2 S_\ell(G)\sum_{\substack{\a \in G \\ i, \b \in B}} S_{k-2} (G| \a) A_{\a\b, i}^2-2 S_{k-1}(G) \sum_{\substack{\a \in G \\ \b, i \in B}} S_{\ell-1}(G|\a) A_{i\a, \b}^2}_{I_5}.
        }
       
	To complete the proof, it remains to show that
	\eq{
	I_1 \sim S_{\ell} (G) \sum_{i \in B} h_{ii} + \sum_{i, \a} F^{\a\a} S_{\ell} (A|i) T_{i\a}-(n-\ell) S_{\ell} (G)  \metric{\frac{D \g}{\g}}{Dh}.
	}
	Using \eqref{def-Tia}, \eqref{Sm-ij delta-ij}, \eqref{Sm-ij A-ij} and \eqref{Sm-ij A^2-ij}, we have
	\eq{
		\sum_{i, \a} F^{\a\a} S_\ell(A| i) T_{i \a}=&~ 
		\frac{1}{\g} \sum_{\a} F^{\a\a} \sum_i S_\ell(A| i) A_{ii}^2- \frac{1}{\g} \sum_{i} S_\ell(A|i) \sum_{\a} F^{\a\a} A_{\a\a}^2\\
		&~+\frac{\psi}{\g} \sum_{\a} F^{\a\a} \sum_{i} S_\ell(A|i) A_{ii} -\frac{\psi}{\g} \sum_i S_\ell (A|i) \sum_{\a} F^{\a\a} A_{\a\a} \\
		&~-\frac{2}{\g^2} \sum_{\a}F^{\a\a} \sum_i S_\ell(A|i) A_{ii} \g_i^2 +\frac{2}{\g^2} \sum_{i} S_\ell (A|i) \sum_{\a} F^{\a\a} A_{\a\a} \g_{\a}^2\\
		=&~ (n-k+1) S_{k-1} (A) \( \frac{1}{\g} (S_1(A) S_{\ell+1} (A) - (\ell+2)S_{\ell+2}(A) ) \right. \\
		&~ \left. +\frac{\psi}{\g} (\ell+1) S_{\ell+1} (A)- \frac{2}{\g^2} \sum_i S_\ell(A| i) A_{ii} \g_i^2 \)\\
		&~-(n-\ell) S_\ell(A)  \( \frac{1}{\g} ( S_1(A) S_{k} (A ) - (k+1) S_{k+1}(A) ) \right. \\
		&~\left.  +\frac{\psi}{\g} k S_k(A) - \frac{2}{\g^2} \sum_{\a}F^{\a\a} A_{\a\a} \g_{\a}^2  \).
	}
	It follows from \eqref{lambda-0 B}, \eqref{Sm=0,m>l} and \eqref{S_l(A|i)=S_(l-1) (A|ij)=0} that 
    \eq{
    S_{\ell+1}(A) \sim  S_{\ell+2}(A) \sim &~0, \\
    S_{\ell} (A|i) A_{ii} \sim &~ 0, \quad \forall ~i.
    }
    Hence, we get
    \eq{
    \sum_{i, \a} F^{\a\a} S_\ell(A| i) T_{i \a}\sim&~-(n-\ell) S_\ell(G)  \( \frac{1}{\g} ( S_1(G) S_{k} (G ) - (k+1) S_{k+1}(G) ) \right. \\
		&~\left.  +\frac{\psi}{\g} k S_k(G) - \frac{2}{\g^2} \sum_{\a}F^{\a\a} A_{\a\a} \g_{\a}^2  \).
    }
    Recall that $F^{\a\a} \sim  S_{k-1} (G| \a)$ for any $\alpha \in G$, and $A_{\a\a} \sim 0$ for any $\alpha \in B$. Also note that $(n-\ell) = \sum_{i \in B} 1$.  Then we have
	\eq{
		&~ S_{\ell} (G) \sum_{i \in B} h_{ii} +\sum_{i, \a} F^{\a\a} S_{\ell} (A|i) T_{i \a}-(n-\ell) S_{\ell} (G)  \metric{\frac{D \g}{\g}}{Dh}\\
		\sim&~ S_{\ell}(G) \sum_{i \in B} \bigg( h_{ii}- \metric{\frac{D \g}{\g}}{Dh} -\frac{\psi}{\varphi} kh  -\frac{1}{\varphi}(S_1(G) h- (k+1) S_{k+1} (G))  \bigg)\\
		&~\quad - S_{\ell}(G) \sum_{i \in B} \bigg( -\frac{2}{\varphi^2} \sum_{\a \in G} S_{k-1} (G| \a) A_{\a\a} \varphi_{\a}^2 \bigg)=I_1.
	}
	This completes the proof of Lemma \ref{lemma-Faa-phiaa}.
    \end{proof}

    In the next step, we estimate $\sum_{\a} F^{\a\a} \phi_{\a\a}$ from above. In the following, we will commute the indices in the $3$-tensor $DA$ several times, which should be handled carefully as $A[\varphi]$ is not a Codazzi tensor by \eqref{s5:A-Codazzi}. Fortunately, we will show that tensor $DA$ is very close to a Codazzi tensor under the equivalence `$\sim$' in many cases.

    By the symmetry of $A[\varphi]$ and \eqref{s5:A-Codazzi}, we have
    \eq{
	A_{i\b, \a} - A_{\a\b, i} =A_{\b i, \a} -A_{\b \a, i} =
    -\frac{\g_{\a}}{\g} A_{\a\a} \delta_{i \b} + \frac{\g_i}{\g} A_{ii} \delta_{\a\b}.
    }
    Then by \eqref{lambda-0 B}, for $i, \a \in B$ we have
    \eq{
	A_{\a\b, i}^2 - A_{i\b, \a}^2 = (A_{\b\a, i} + A_{\b i, \a})  \bigg(-\frac{\g_i}{\g} A_{ii} \delta_{\b\a}+\frac{\g_{\a}}{\g} A_{\a\a} \delta_{\b i} \bigg) \sim 0.
	\label{A2abi-A2iba sim 0}
    }
    Similarly, for $i, \b \in B$, we have
    \eq{\label{Aabi-Aiba sim 0}
	A_{\a\b, i}^2 - A_{i\a, \b}^2 = A_{\a\b, i}^2 - A_{\a i, \b}^2 \sim 0.
    }
    For $i \in B$, $\a, \b \in G$ with $ \a \neq \b$, we have $\delta_{\a\b} = 0$, $\delta_{i \b} =0$ and then
    \eq{\label{Aabi=Aiba}
	A_{\a\b, i}^2 - A_{i\b, \a}^2=0.
    }
    For $i \in B$ and $\a \in G$, using $\delta_{\a i} = 0$ and $A_{ii} \sim 0$, we have
    \eq{\label{Aaia-Aaai sim 0}
	A_{\a i, \a} - A_{\a\a, i} = -\frac{\g_{\a}}{\g} A_{\a\a} \delta_{\a i} +\frac{\g_i }{\g} A_{ii} \delta_{\a\a} \sim 0.
    }

\begin{lemma}
	\eq{ \label{Faa-phiaa-lem5.4}
	&\sum_{\a} F^{\a\a} \phi_{\a\a} \\
 \lesssim &~   S_{\ell}(G) \sum_{i \in B} \( h_{ii}-\metric{\frac{D \g}{\g}}{Dh} -\frac{\psi}{\varphi} kh - \frac{1}{\varphi}(S_1(G) h- (k+1) S_{k+1} (G))  \)\\
    &~- \sum_{\substack{i \in B\\ \a, \b \in G\\ \a \neq \b}} S_\ell(G)
	 S_{k-2} (G|\a\b) A_{\a\a, i} A_{\b\b, i}-  \sum_{\substack{\a \in G\\ i,\b \in B}} S_{\ell-1}(G) S_{k-1}(G|\a) A_{i\b, \a}^2\\
	&~-\sum_{\substack{i \in B \\ \a, \b \in G \\ \a \neq \b}} S_{\ell-1} (G|\b) S_{k-1}(G| \a) A_{i\b, \a}^2- 2\sum_{\substack{i \in B\\ \a \in G}} S_{\ell-1} (G| \a) S_{k-1} (G| \a) A_{\a \a,i}^2.
	}
\end{lemma}	
		
\begin{proof}
        By Lemma \ref{lemma-Faa-phiaa}, it suffices to estimate the terms $I_i (i=1,2,\ldots,5)$ from above. First, we estimate $I_2$ given by \eqref{def-I2}. Using \eqref{Aaia-Aaai sim 0} we get
	\eq{ \label{I_2-upper-bound}
	I_2 \sim& - \sum_{\substack{i \in B\\ \a, \b \in G\\ \a \neq \b}} S_\ell(G) 
	S_{k-2} (G|\a\b) A_{\a\a, i} A_{\b\b, i}-\sum_{\substack{\a \in G\\ i,\b \in B}} S_{\ell-1}(G) S_{k-1}(G|\a) A_{i\b, \a}^2\\
	&-\sum_{\substack{i \in B \\ \a, \b \in G \\ \a \neq \b}} S_{\ell-1} (G| \b) S_{k-1} (G| \a) A_{i\b, \a}^2-\sum_{\substack{i \in B \\ \a \in G}} S_{\ell-1} (G| \a) S_{k-1} (G|\a) A_{\a\a, i}^2.
	}
	Now, we estimate $I_3$ given by \eqref{def-I3}. Since $k \leq \ell$, using the (non-sharp) Newton-MacLaurin inequality
        \eq{
        S_\ell(G)S_{k-2}(G)\leq S_{\ell-1}(G)S_{k-1}(G)
        }
	and \eqref{A2abi-A2iba sim 0}, we have
	\eq{ \label{I_3-upper-bound}
	I_3 \leq S_{\ell-1}(G) S_{k-1} (G) \sum_{i,\a, \b \in B} (A_{\a\b,i}^2-A_{i \b, \a}^2) \sim 0.
	}
	Next, we estimate $I_4$ given by \eqref{def-I4}.  For any $\alpha, \beta \in G$ with $\a \neq \b$, we have 
	\eq{
		S_{\ell-1} (G|\b) S_{k-1} (G|\a) &= S_{\ell-1}(G|\b)(S_{k-1}(G|\a\b)+A_{\b\b}S_{k-2}(G|\a\b)) \\
         &\geq 	S_{\ell-1} (G|\b) A_{\b\b} S_{k-2} (G|\a\b) \\
         &= S_{\ell} (G) S_{k-2} (G|\a\b).
	}
	Using \eqref{Aabi=Aiba} and \eqref{Aaia-Aaai sim 0}, we get
	\eq{ \label{I_4-upper-bound}
	I_4=& \sum_{\substack{i \in B \\ \a,\b \in G\\ \a \neq \b}} S_\ell (G) S_{k-2} (G| \a\b) A_{\a\b, i}^2
	-\sum_{\substack{i \in B \\ \a, \b \in G}} S_{\ell-1} (G|\b) S_{k-1}(G| \a) A_{i\b, \a}^2\\
	\leq& \sum_{\substack{i \in B \\ \a,\b \in G \\ \a \neq \b}} S_{\ell-1} (G|\b) S_{k-1} (G| \a) (A_{\a\b, i}^2-A_{i\b, \a}^2)
	-  \sum_{\substack{i \in B \\ \a \in G}} S_{\ell-1} (G| \a) S_{k-1} (G| \a) A_{i\a, \a}^2\\
	\sim &-  \sum_{\substack{i \in B\\ \a \in G}} S_{\ell-1} (G| \a) S_{k-1} (G| \a) A_{\a \a,i}^2.
	}
	Then, we estimate $I_5$ given by \eqref{def-I5}. For any $\a \in G$, we have
	\eq{
		S_\ell(G)S_{k-2} (G| \a) &= S_{\ell-1}(G|\a) A_{\a\a} S_{k-2} (G| \a) \\
        &\leq S_{\ell-1}(G|\a) (S_{k-1}(G|\a)+A_{\a\a} S_{k-2} (G| \a))\\
        &= S_{\ell-1}(G|\a) S_{k-1}(G).
	}
	This together with \eqref{Aabi-Aiba sim 0}, \eqref{s5:A-Codazzi} and \eqref{lambda-0 B} gives
    \eq{ \label{I5-upper-bound}
	I_5 \sim& -2 \sum_{\substack{\a \in G\\ i, \b \in B}} \( S_{\ell-1} (G| \a) S_{k-1} (G) - S_{\ell}(G)  S_{k-2} (G| \a)\) A_{i\a, \b}^2\\
	\lesssim& - 2\sum_{\substack{\a \in G\\ i \in B}} \( S_{\ell-1} (G| \a) S_{k-1} (G) - S_{\ell} (G) S_{k-2} (G| \a) \) A_{i\a,i}^2\\
	\sim& -2 \sum_{\substack{\a \in G\\ i \in B}} \( S_{\ell-1} (G| \a) S_{k-1} (G) - S_{\ell} (G) S_{k-2} (G| \a) \) \( A_{ii, \a}^2+ 2 \frac{\varphi_{\a} }{\varphi} A_{\a\a} A_{ii, \a} + \frac{\varphi_{\a}^2}{\varphi^2 } A_{\a\a}^2 \)\\
	\lesssim& -2 \sum_{\substack{\a \in G\\i \in B}}  S_{\ell} (G) \(S_{k-1} (G)- S_{k-2} (G| \a) A_{\a\a} \) \frac{\varphi_{\a}^2}{\varphi^2} A_{\a\a}\\
	=& -2\sum_{\substack{\a \in G\\ i \in B}} S_{\ell}(G) S_{k-1} (G| \a)   \frac{\varphi_{\a}^2}{\varphi^2} A_{\a\a},
    }
    where in the second inequality we used \eqref{Aii,a-0} and 
    \eq{
    (S_{\ell-1}(G|\a) S_{k-1}(G)-S_{\ell}(G)S_{k-2}(G|\a))A_{\a\a}=S_\ell(G)(S_{k-1}(G)-A_{\a\a}S_{k-2}(G|\a)).
    }
    By \eqref{def-I1} and \eqref{I5-upper-bound}, we get
    \eq{\label{I_1+I5-upper-bound}
             I_1+I_5 
    \lesssim ~S_{\ell}(G) \sum_{i \in B} \( h_{ii} -\metric{\frac{D \g}{\g}}{Dh}-\frac{\psi}{\varphi} kh  -\frac{1}{\varphi}(S_1(G) h- (k+1) S_{k+1} (G))  \).
    }
    Finally, the desired estimate \eqref{Faa-phiaa-lem5.4} follows from \eqref{I_2-upper-bound}, \eqref{I_3-upper-bound}, \eqref{I_4-upper-bound} and \eqref{I_1+I5-upper-bound}.
\end{proof}	

The following estimates were proved in \cite{GM03}. 
\begin{lemma} 
	\eq{\label{GM-Alg-lem}
		- \sum_{\substack{i \in B\\ \a, \b \in G\\ \a \neq \b}} S_\ell (G)   S_{k-2} (G| \a\b) A_{\a\a, i} A_{\b\b i} 
	\lesssim&~
	-S_\ell (G)\sum_{i \in B} \frac{h_{i}^2}{h} + \sum_{\substack{i \in B\\ \a \in G}} S_{\ell-1} (G| \a) S_{k-1} (G |\a )A_{\a\a, i}^2,\\
	- \sum_{\substack{i \in B\\ \a \in G}} S_{\ell-1}(G| \a) S_{k-1} (G| \a) A_{\a\a, i}^2 \lesssim&~  -	S_\ell(G)\frac{1}{k} \sum_{i \in B}\frac{h_i^2}{h}.
	}
\end{lemma}
\begin{proof}
The estimates \eqref{GM-Alg-lem} are exactly the claims $I_2 \lesssim 0$ and $I_3 \lesssim 0$ in \cite[p. 571]{GM03}.
\end{proof}
Inserting the estimates \eqref{GM-Alg-lem} into \eqref{Faa-phiaa-lem5.4}, using the Newton-MacLaurin inequality
    \eq{
    S_1(G) h-(k+1) S_{k+1} (G)\sim  S_1(G) S_k(G) -(k+1) S_{k+1} (G)\geq 0,
    }
and 
\eq{
- \sum_{\substack{\a\in G \\ i,\b\in B}} S_{\ell-1}(G)S_{k-1}(G|\a)A_{i\b,\a}^2\leq &~0,\\
-\sum_{\substack{i \in B \\ \a, \b \in G \\ \a \neq \b}} S_{\ell-1} (G|\b) S_{k-1}(G| \a) A_{i\b, \a}^2\leq &~ 0,
}
we obtain the following useful estimate.
\begin{lemma} \label{lem-Faaphiaa-final-h}
	\eq{
		\sum_{\a}F^{\a\a} \phi_{\a\a} \lesssim &~ S_\ell(G) \sum_{i \in B} \bigg(h_{ii} -\metric{\frac{D \g}{\g}}{Dh}
  -\frac{k+1}{k} \frac{h_i^2}{h} -\frac{\psi}{\g} k h \bigg).
	}
\end{lemma}	

    \begin{lemma}\label{lem-convexity of h->f}
    Fix $i \in B$. Then
    \eq{
          &~ h_{ii} -\metric{\frac{D \g}{\g}}{Dh}-\frac{k+1}{k} \frac{h_i^2}{h} -\frac{\psi}{\g} k h\\
       \sim &~ -k \varphi^{n+p-k} f^{\frac{k+1}{k} } \left\{ (f^{-\frac{1}{k} } )_{ii} +\frac{2(k-p-n)}{k} (f^{-\frac{1}{k} })_i \frac{\varphi_i}{\varphi} - \metric{D f^{-\frac{1}{k} }}{\frac{D \varphi}{\varphi} }  \right. \\
        &~\left. -\( \frac{(k-p-n)(n+p)}{k^2} \frac{\varphi_i^2}{\varphi^2} - \frac{n+p}{2k} \frac{|D \varphi|^2}{\varphi^2} - \frac{n+p}{2k} - \frac{2k-p-n}{2k} \frac{1}{\varphi^2} \)f^{ -\frac{1}{k} } \right\}.
    }
    \end{lemma}

    \begin{proof}
   Recall equation \eqref{eq-in deformation lemma},  \eq{ \label{un-normalized-horo-p-CM-problem}
    F=S_k(A)=h=\g^{-t} f, \quad t=k-p-n.
    }
    Fix $i \in B$. Taking the covariant derivatives of $h = \g^{-t} f$, we get
    \eq{
	h_i =&~ -t \varphi^{-t-1}  \varphi_i f+ \varphi^{-t} f_i, \\
	h_{ii} =&~t (t+1) \varphi^{-t-2} \varphi_i^2 f-2t \varphi^{-t-1} \varphi_i f_i -t \varphi^{-t-1} f \varphi_{ii} +\varphi^{-t} f_{ii},\\
  \frac{h_i^2}{h} =&~ t^2\varphi^{-t-2} \varphi_i^2 f +\varphi^{-t} \frac{f_i^2}{f} -2t \varphi^{-t-1}f_i\varphi_i.
    }
    Since $\lambda_i \sim 0$ for $i \in B$ by \eqref{lambda-0 B}, we have
    \begin{equation*}
	\varphi_{ii} \sim \frac{1}{2} \frac{|D \varphi|^2}{\varphi} - \frac{1}{2}\(\varphi -\frac{1}{ \varphi } \).
    \end{equation*}
    Then we get
    \eq{ \label{estimte-RHS-deformation-lemma}
    &~h_{ii} - \frac{k+1}{k} \frac{h_i^2}{h} -\frac{\metric{Dh}{D \varphi}}{\varphi}- \frac{\psi}{\varphi} kh\\
    =&~t (t+1) \varphi^{-t-2} \varphi_i^2 f-2t \varphi^{-t-1} \varphi_i f_i -t \varphi^{-t-1} f \varphi_{ii} +\varphi^{-t} f_{ii}\\
    &~- \frac{k+1}{k} t^2 \varphi^{-t-2} \varphi_i^2f -\frac{k+1}{k} \varphi^{-t}\frac{f_i^2}{f}+2 \frac{k+1}{k} t\varphi^{-t-1} f_i \varphi_i\\
    &~- \sum_{j=1}^n \frac{\varphi_j}{\varphi} \(-t \varphi^{-t-1} \varphi_j f+ \varphi^{-t} f_j \)- \( \frac{1}{2} \frac{|D \varphi|^2}{\varphi^2} +\frac{1}{2} (1+ \frac{1}{\varphi^2}) \) k \varphi^{-t} f \\
    =&~\varphi^{-t} f_{ii}- \frac{k+1}{k} \varphi^{-t} \frac{f_i^2}{f} 
    +\frac{2t}{k} \varphi^{-t-1} \varphi_i f_i-\varphi^{-t-1} \metric{Df}{D \varphi} - t\varphi^{-t-1} f \varphi_{ii}\\
    &~+ \frac{t(k-t)}{k} \varphi^{-t-2} \varphi_i^2 f 
    + (t- \frac{k}{2}) \varphi^{-t-2} |D \varphi|^2 f- \frac{k}{2}(1+ \frac{1}{\varphi^2})\varphi^{-t} f\\
    \sim&~\varphi^{-t} \left\{  f_{ii} - \frac{k+1}{k} \frac{f_i^2}{f} + \frac{2t}{k} \frac{f_i \varphi_i}{\varphi}- \metric{Df}{\frac{D\varphi}{\varphi}}  \right.\\
    &~\left. + \bigg( \frac{t(k-t)}{k} \frac{\varphi_i^2}{\varphi^2} -\frac{k-t}{2} \frac{|D \varphi|^2}{\varphi^2}- \frac{k-t}{2} - \frac{k+t}{2} \frac{1}{\varphi^2} \bigg) f \right\}.
    }
    Using the relation
\begin{align*}
f_i =-k f^{\frac{k+1}{k}} (f^{-\frac{1}{k}})_{i}, \quad
f_{ii} - \frac{k+1}{k} \frac{f_i^2}{f} = -k f^{ \frac{k+1}{k} } (f^{-\frac{1}{k}} )_{ii},
\end{align*}
we rewrite the estimate \eqref{estimte-RHS-deformation-lemma} as 
\begin{align*}
&~h_{ii} - \frac{k+1}{k} \frac{h_i^2}{h} -\frac{\metric{Dh}{D \varphi}}{\varphi}- \frac{\psi}{\varphi} kh\\
\sim&~-k \varphi^{-t} f^{\frac{k+1}{k} }\left\{ (f^{-\frac{1}{k} } )_{ii} +\frac{2t}{k} \frac{(f^{-\frac{1}{k}})_i\varphi_i}{\varphi} - \metric{D f^{-\frac{1}{k}}}{\frac{D \varphi}{\varphi} }  \right.\\
&~\left. -\bigg( \frac{t(k-t)}{k^2} \frac{\varphi_i^2}{\varphi^2} - \frac{k-t}{2k} \frac{|D \varphi|^2}{\varphi^2} - \frac{k-t}{2k} - \frac{k+t}{2k} \frac{1}{\varphi^2} \bigg)f^{ -\frac{1}{k} } \right\}\\
=&~-k \varphi^{n+p-k} f^{\frac{k+1}{k} } \left\{ (f^{-\frac{1}{k} } )_{ii} +\frac{2(k-p-n)}{k}  (f^{-\frac{1}{k} })_i\frac{\varphi_i}{\varphi} - \metric{D f^{-\frac{1}{k}}}{\frac{D \varphi}{\varphi} }  \right. \\
&~\left. -\bigg( \frac{(k-p-n)(n+p)}{k^2} \frac{\varphi_i^2}{\varphi^2} - \frac{n+p}{2k} \frac{|D \varphi|^2}{\varphi^2} - \frac{n+p}{2k} - \frac{2k-p-n}{2k} \frac{1}{\varphi^2} \bigg)f^{ -\frac{1}{k} } \right\},
\end{align*}
where we used $t=k-p-n$ in the last equality. We complete the proof of Lemma \ref{lem-Faaphiaa-final-h}.
\end{proof}

\begin{proof}[Proof of Lemma \ref{deformation-lemma} (Deformation lemma)]
    It follows from Lemma \ref{lem-Faaphiaa-final-h} and Lemma \ref{lem-convexity of h->f}.
\end{proof}

\section{Full rank theorem}
We have proved an upper bound of the matrix $A[\varphi]$ of solutions to the horospherical $p$-Christoffel-Minkowski problem \eqref{s1:horo-p-CM-problem} in Lemma \ref{lem-phi C2}. Different from the prescribed $p$-shifted Weingarten curvature problem, there is no such lower bound of the matrix $A[\varphi]$ for the horospherical $p$-Christoffel-Minkowski problem in general. The main purpose of this section is to impose suitable structural assumptions on the prescribed function $f$ such that any (weakly) h-convex, even solution $\varphi$ to the horospherical $p$-Christoffel-Minkowski problem \eqref{s1:horo-p-CM-problem} is actually strictly h-convex, in another word, $A[\varphi]$ is of full rank. 
\begin{theorem}[Full rank theorem]\label{thm-full-rank}
	Let $f$ be a smooth, positive, even function on $\mathbb{S}^n$ satisfying Assumption \ref{Assump-f}. Let $\varphi$ be a $C^4$ smooth, even solution to equation \eqref{s1:horo-p-CM-problem} with positive semi-definite $A [\varphi]$ on $\mathbb{S}^n$. Then $A[\varphi]>0$ on $\mathbb{S}^n$.
\end{theorem}

To prove Theorem \ref{thm-full-rank}, we first use the deformation lemma \ref{deformation-lemma} to show a differential inequality for solutions to \eqref{s1:horo-p-CM-problem} on $\bbS^n$ when $f$ satisfies Assumption \ref{Assump-f}.
\begin{lemma}\label{lem-defor-lemma-weaker}
    Suppose that $\varphi \in C^4(\bbS^n)$ is an even solution of the horospherical $p$-Christoffel-Minkowski problem \eqref{s1:horo-p-CM-problem} with $A[\varphi]$ positive semi-definite, and $p_{\ell} (A[\varphi]) \geq C_0$ for some integer $n-k \leq \ell \leq n-1$. Assume that $f$ is even on $\bbS^n$ and satisfies Assumption \ref{Assump-f}. Then there are constants $C_1$, $C_2$ depending only on $|| \varphi||_{C^3}$, $||f||_{C^2}$, $n$, $k$, $p$ and $C_0$ such that the differential inequality
	\eq{
		\sum_{\a, \b=1}^n p_{n-k}^{\a\b} \phi_{\a\b}(z) \leq   C_1 |D \phi (z)| +C_2 \phi (z)
	}
 holds on $\bbS^n$, where $\phi(z) = p_{\ell+1} (A[\varphi])$.
\end{lemma}

\begin{proof}
    Comparing \eqref{s1:horo-p-CM-problem} with equation \eqref{eq-in deformation lemma} in the deformation lemma \ref{deformation-lemma}, it suffices to prove that
    \eq{\label{deformation-cond on f-in full rank}
    	&~D^2 (f^{-\frac{1}{n-k} } ) -\frac{2(k+p)}{n-k} d(f^{-\frac{1}{n-k} }) d(\log \varphi) - \metric{D f^{-\frac{1}{n-k} }}{D \log \varphi} \sigma\\
		&~+  \frac{(k+p)(n+p)}{(n-k)^2} f^{ -\frac{1}{n-k} } (d \log \varphi)^2
		+ \(\frac{n+p}{2(n-k)} \frac{|D \varphi|^2}{\varphi^2}  + \frac{n+p}{2(n-k)} + \frac{n-2k-p}{2(n-k)} \frac{1}{\varphi^2} \)f^{ -\frac{1}{n-k} } \sigma
    }
    is positive semi-definite on $\bbS^n$.
    
	When $p>-n$, the matrix \eqref{deformation-cond on f-in full rank} is exactly  \cite[Eq. (7.39)]{LX22} (up to $\varphi^2>0$), and we can use \cite[Lem. 7.6 \& Assump. 7.1]{LX22} to find a sufficient condition on $f$ so that \eqref{deformation-cond on f-in full rank} is positive semi-definite for even solution $\varphi$, which are (2), (3), (4) and (5) of Assumption \ref{Assump-f}. When $p=-n$ and $k=n-1$, such condition of $f$ was provided in  \cite[Lem. 4.3]{Che24}, and it is the case $k=n-1$ of (1) in Assumption \ref{Assump-f}. 
 
    Here, we only need to deal with the case $p=-n$ and $1 \leq k \leq n-2$ in Assumption \ref{Assump-f}. In this case, $q=\frac{n+p}{n-k}-1=-1$. By the $C^0$-estimate of equation \eqref{s1:horo-p-CM-problem} in Lemma \ref{lem-C0-est}, we have
	\eq{
		\xi_{-1} (\varphi_{\min} ) \geq 
    \(f_{\max} \)^{- \frac{1}{n-k}}.
	}
	Since $\xi_{-1} (t) = 2t^{-1 } \( t- \frac{1}{t} \)^{-1}$, we have
    \eq{
		\varphi_{\min} \leq  \( 1+ 2 \( f_{\max} \)^{\frac{1}{n-k} } \)^{\frac{1}{2} }.
	}
	Then we can apply \eqref{eq-phi max < phi min} to get
 	\eq{
		\varphi_{\max}^2 \leq \bigg(
		\bigg( 1+ 2 (f_{\max})^{\frac{1}{n-k} } \bigg)^{\frac{1}{2} }
		+ \bigg( 2 (f_{\max})^{\frac{1}{n-k} }  \bigg)^{\frac{1}{2} }
		\bigg)^2 
		< 2+ 8  ( f_{\max})^{\frac{1}{n-k}}.
	}
	This together with $p=-n$, \eqref{eq-Dphi<phi} and (1) in Assumption \ref{Assump-f} gives
 	\eq{
	&~D^2(f^{-\frac{1}{n-k} } ) +2 d(f^{-\frac{1}{n-k} }) d\log \varphi - \metric{D f^{-\frac{1}{n-k} }}{D \log \varphi} \sigma+\frac{1}{\varphi^2}f^{ -\frac{1}{n-k} }  \sigma\\
		>&~D^2(f^{-\frac{1}{n-k} } )-3 |D f^{-\frac{1}{n-k} }| \sigma + \frac{f^{-\frac{1}{n-k}}}{2 + 8 (f_{\max})^\frac{1}{n-k} }\sigma \geq 0 .
	}

	Consequently, \eqref{deformation-cond on f-in full rank} is positive semi-definite if $p=-n$ and $f$ satisfies (1) in Assumption \ref{Assump-f}.
	This completes the proof of Lemma \ref{lem-defor-lemma-weaker}.
\end{proof}

Using the shifted Minkowski formula for closed hypersurfaces in $\mathbb{H}^{n+1}$ (see \cite{HLW22}), Chen \cite{Che24} proved the following formula for smooth positive function $\varphi$ on $\mathbb{S}^n$ with $A[\varphi]>0$, which can be easily extended to the case $A[\varphi] \geq 0$.
\begin{lemma}
	Let $\varphi$ be a smooth, positive function on $\mathbb{S}^n$ with $A[\varphi] \geq 0$. For $\ell=0,1, \ldots, n-1$, there holds
	\eq{\label{shifted-Mink}
	&\int_{\mathbb{S}^n} \varphi^{\ell-n} p_{\ell+1} (A[\varphi]) d\sigma
	= \int_{\mathbb{S}^n} \(\frac{|D \varphi|^2}{2\varphi} +\frac{1}{2} \( \varphi -\varphi^{-1} \)  \) \varphi^{\ell-n} p_\ell \( A[\varphi]  \) d\sigma,  
}
where $d\sigma$ is the area element of $\mathbb{S}^n$.
\end{lemma}
\begin{proof}
    For convenience of the reader, we give the proof here. We first consider the case for smooth positive function $\varphi$ on $\bbS^n$ with $A[\varphi]>0$. In this case, the strictly h-convex hypersurface $\cM$ in $\bbH^{n+1}$ can be recovered from $\varphi$ via the embedding \eqref{X(z)}, which is the inverse of the horospherical Gauss map $G:\bbS^n\ra \cM$. Recall that the shifted Minkowski formula (\cite[Lem. 2.6]{HLW22}) for closed hypersurfaces in hyperbolic space $\bbH^{n+1}$:
    \eq{\label{shifted-Minkowski-formula}
    \int_{\cM} (\cosh r-\tilde u) p_m(\tilde\kappa) d\mu=\int_{\cM} \tilde u p_{m+1}(\tilde\kappa) d\mu, \quad m=0,1,\ldots,n-1,
    }
    where $\tilde\kappa=(\kappa_1-1,\ldots,\kappa_n-1)$ are the shifted principal curvatures, and $\tilde u=\bar g(\sinh r\partial_r,\nu)$ is the support function of $\cM$.
    %, respectively. 
    To transform the shifted Minkowski formula \eqref{shifted-Minkowski-formula} into the formula \eqref{shifted-Mink}, we have (\cite[Lem. 3.9]{LX22})
    \eq{
    \cosh r\circ G^{-1} &=\frac{|D\varphi|^2}{2\varphi}+\frac{1}{2}(\varphi+\varphi^{-1}), \\
    \tilde u \circ G^{-1}&=\frac{|D\varphi|^2}{2\varphi}+\frac{1}{2}(\varphi-\varphi^{-1}).
    }
    Then we get
    \eq{
    (\cosh r-\tilde u)\circ G^{-1}=\varphi^{-1}.
    }
    By \cite[Cor. 2.1]{LX22}, we have 
    \eq{
    d\mu \circ G^{-1}=p_n(A[\varphi]) d\sigma.
    }
    Note that the matrix $\(h_i{}^j(G^{-1}(z))-\delta_i{}^j\)$ is the inverse matrix of $(\varphi(z)\sigma^{jl}A_{li}[\varphi(z)])$, then the eigenvalues $\tilde\kappa$ of $\(h_i{}^j(G^{-1}(z))-\delta_i{}^j\)$ are the reciprocal of the corresponding eigenvalues of $(\varphi(z)\sigma^{jl}A_{li}[\varphi(z)])$. 
    So we obtain
    \eq{
    (p_{n}(\tilde \kappa)d\mu)\circ G^{-1} = \varphi^{-n} d\sigma,
    }
    and thus
    \eq{
    \varphi^{\ell-n}p_{\ell+1}(A[\varphi])d\sigma&=\((\cosh r-\tilde u)p_{n-\ell-1}(\tilde \kappa) d\mu\)\circ G^{-1}, \\
    \(\frac{|D \varphi|^2}{2\varphi} +\frac{1}{2} \( \varphi -\varphi^{-1} \)  \) \varphi^{\ell-n} p_\ell \( A[\varphi]  \)d\sigma &=\(\tilde u p_{n-\ell}(\tilde \kappa)d\mu\)\circ G^{-1}.
    }
    Combining these with the shifted Minkowski formula \eqref{shifted-Minkowski-formula}, we obtain the formula \eqref{shifted-Mink}.
   
    For the case for smooth positive function $\varphi$ on $\bbS^n$ with $A[\varphi]\geq 0$, we approximate $\varphi$ by $\varphi+\varepsilon$ with $\varepsilon>0$. Note that $A[\varphi] \geq 0$ implies that $A[\varphi+\varepsilon]>0$ for any $\varepsilon>0$, since
    \eq{
    A_{ij}[\varphi+\varepsilon] =&~
    (\varphi+\varepsilon)_{ij} - \frac{1}{2} \frac{|D (\varphi+\varepsilon)|^2}{\varphi+\varepsilon} \sigma_{ij} + \frac{1}{2} \(\varphi+\varepsilon - (\varphi+\varepsilon)^{-1} \) \sigma_{ij}\\
    \geq&~ \varphi_{ij} - \frac{1}{2} \frac{|D \varphi|^2}{\varphi} \sigma_{ij} +\frac{1}{2} (\varphi - \varphi^{-1}) \sigma_{ij} +\frac{\varepsilon}{2}\sigma_{ij}\\
    =&~ A_{ij} [\varphi] + \frac{\varepsilon}{2} \sigma_{ij}.
    }Hence, one can prove \eqref{shifted-Mink} for $\varphi +\varepsilon$  first and then let $\varepsilon \to 0^+$ to obtain \eqref{shifted-Mink} for $\varphi$ with $A[\varphi]\geq 0$. 
\end{proof}

Now, we prove Theorem \ref{thm-full-rank}.

\begin{proof}[Proof of Theorem \ref{thm-full-rank} (Full rank theorem)]
	If $A[\varphi(z)]$ is not of full rank at some $z_0  \in \mathbb{S}^n$, then there is $n-k \leq \ell \leq n-1$ such that $p_{\ell} (A[\varphi]) >0$ on $\mathbb{S}^n$, and $\phi (z_0) =S_{\ell+1} (A[\varphi(z_0)]) =0$. Then by Lemma \ref{deformation-lemma}, we have
	\eq{
	\sum_{\a, \b=1}^n p_{n-k}^{\a\b} (z) \phi_{\a\b} (z) \leq C_1 |D \phi(z)| +C_2 \phi(z).
	}
	The strong minimum principle implies $\phi = p_{\ell+1} (A[\varphi]) \equiv 0$. Then by formula \eqref{shifted-Mink}, we have $\varphi \equiv 1$. Hence $A[\varphi]=0$, which is a contradiction to equation \eqref{s1:horo-p-CM-problem}.
\end{proof}

\section{Proofs of Theorems \ref{thm-horo-p-CM-problem}, \ref{thm-Lp-Weingarten-eq-hyperbolic} and \ref{thm-horo-p-Minkowski}}
We first recall the following uniqueness result.
\begin{lemma}\label{s6:lem-uniqueness}
\noindent 
\begin{enumerate}
    \item (\cite[Prop. 8.1 \& Thm. 8.1]{LX22}) Let $n\geq 1$. For any $p \geq -n$, the even, h-convex solutions to equation
	\eq{ \label{s6:general-eq}
		F(A[\varphi]) =\varphi^{\frac{n+p}{n-k}-1} \gamma^{\frac{1}{n-k}}
	}
	are constant, where $\gamma>0$ is a constant, and $F=p_{n-k}^{1/(n-k)}$ or $F=(p_n/p_k)^{1/(n-k)}$.
     \item (\cite[Thm. 1.2]{LW24}) Let $n=1$. For any $-7\leq p<-1$, the even, h-convex solutions to equation
	\eq{ 
		\varphi_{\theta\theta}-\frac{1}{2}\frac{\varphi_\theta^2}{\varphi}+\frac{\varphi-\varphi^{-1}}{2} =\varphi^{p} \gamma
	}
	are constant, where $\gamma>0$.
\end{enumerate}
   
\end{lemma}

Recall that the function $\xi_q: (1, \infty) \to (0, \infty)$ defined by \eqref{def-zeta-q}, %that is,
\eq{
\xi_q(t)=2 t^{q} (t-t^{-1})^{-1}, \quad q=\frac{n+p}{n-k}-1.
}
Assume that $\varphi\equiv c_0$ is a constant solution to \eqref{s6:general-eq}. Then $A[c_0]=\frac{1}{2}(c_0-c_0^{-1})$ and
\eq{ \label{s2:c0-eq}
\xi_q(c_0)=\gamma^{-\frac{1}{n-k}} . 
}
By the calculation in Lemma \ref{lem-C0-est} and Lemma \ref{s6:lem-uniqueness}, we have 
\begin{itemize}
		\item Either $-1\leq q<1$ (that is, $-n \leq p<n-2k$, when $n\geq 2$ and $0\leq k\leq n-1$), or $-7\leq q<1$ (that is, $-7 \leq p<1$, when $n=1$ and $k=0$): There exists a unique solution to \eqref{s2:c0-eq} for any $\gamma>0$;
        \item $q=1$ (that is, $p=n-2k$, when $n\geq 1$ and $0\leq k\leq n-1$): There exists a unique solution to \eqref{s2:c0-eq} for any $\gamma^{- \frac{1}{n-k}}>2$;
            \item $q>1$ (that is, $p>n-2k$, when $n\geq 1$ and $0\leq k\leq n-1$): If $\gamma^{- \frac{1}{n-k}}  = \frac{(q+1)^{\frac{q+1}{2}}}{(q-1)^{\frac{q-1}{2}}}$, then there exists a unique solution to \eqref{s2:c0-eq};
            If $\gamma^{- \frac{1}{n-k}} > \frac{(q+1)^{\frac{q+1}{2}}}{(q-1)^{\frac{q-1}{2}}} $, then there exist exactly two distinct solutions to \eqref{s2:c0-eq}.
\end{itemize}

\begin{remark}
	When $F = p_{n-k}^{1/(n-k)}$, equation $\varphi^{1-\frac{n+p}{n-k}} F( A[\varphi]) = \gamma^{ \frac{1}{n-k} }$ is equivalent to 
	\begin{align*}
		\varphi^{-k-p} p_{n-k} (A [ \varphi]) = \gamma,
	\end{align*}
	and when $F = (p_n/p_{k})^{1/(n-k)}$, it is equivalent to
	\begin{align*}
	\gamma \varphi^{n+p} p_{n-k} ( \kappa -1) =1.
	\end{align*}
\end{remark}

\begin{proof}[Proof of Theorem \ref{thm-horo-p-CM-problem}] 
To establish the existence of the solution, we use the degree theory for second order fully nonlinear elliptic operators developed by Li \cite{Li89}. If $-1\leq q<0$ (that is, $-n\leq p<-k$), we take
\eq{
f_t =\( (1-t)(f_{\max})^{-\frac{1}{n-k} } + t f^{-\frac{1}{n-k} }\)^{-(n-k)}, \quad t \in [0,1];
}
If $q\geq 0$ (that is, $p\geq -k$), we take
\eq{
f_t =\( (1-t)(f_{\max})^{-\frac{1}{n+p}}+ t f^{-\frac{1}{n+p}}\)^{-(n+p)}, \quad t \in [0,1].
}

Consider the Banach space
\eq{
	\mathcal{B}^{2, \alpha}(\mathbb{S}^n )  : = \{ w \in C^{2, \alpha}(\mathbb{S}^n ) ~|~ \text{$w$ is even} \}
}
and
\eq{
	\mathcal{B}_0^{4, \alpha}(\mathbb{S}^n ) := \{w \in C^{4, \alpha}(\mathbb{S}^n ) ~|~ \text{$w$ is  even}\}.
}
Define $\mathcal{L}_t:\mathcal{B}_0^{4,\alpha}(\mathbb{S}^n )\ra \mathcal{B}^{2,\alpha}(\mathbb{S}^n )$ by
\begin{align*}
\mathcal{L}_t (\g) =p_{n-k}^\frac{1}{n-k}(A[\g])-\g^{q} f_t^\frac{1}{n-k}.
\end{align*}	

Let $R>0$ be a constant. For $-1\leq q\leq 1$, let
\eq{
\mathcal{O}_R:= \left\{ w \in \mathcal{B}_0^{4, \alpha}(\bbS^n) ~|~  \| w  \|_{C^{4, \alpha}}<R, \ w >1 +\frac{1}{R}, \ A[w] > 0,
\ p_{n-k} (A [w]) > \frac{1}{R} \right\};
}
and for $q>1$, let
\eq{
\mathcal{O}_R:= \left\{ w \in \mathcal{B}_0^{4, \alpha}(\bbS^n) ~|~  \| w  \|_{C^{4, \alpha}}<R, \ w >\sqrt\frac{q+1}{q-1}, \ A[w] > 0,
\ p_{n-k} (A [w]) > \frac{1}{R} \right\}.
}
Then $\mathcal{O}_R$ is an open bounded subset of $\mathcal{B}_{0}^{4,\alpha}(\mathbb S^n)$ for all $q\geq -1$. Moreover, $\cL_t$ is uniformly elliptic in $\cO_R$ for any $t\in [0,1]$.

{\bf Claim}: For each $t\in [0,1]$, equation $\cL_t(\varphi)=0$ admits no solution on $\partial \cO_R$ if $R>0$ is sufficiently large. 

First of all, we show that if $f$ satisfies Assumption \ref{Assump-f} and Assumption \ref{Assump-barrier-f}, then $f_t$ satisfies these assumptions for all $t \in [0,1]$. For the former case $-1\leq q<0$ (i.e. $-n\leq p<-k$), it follows from $f\leq f_{\max}$ that  
\eq{
(f_t)^{-\frac{1}{n-k}}= (1-t)(f_{\max})^{-\frac{1}{n-k}}+tf^{-\frac{1}{n-k}} \geq (f_{\max})^{-\frac{1}{n-k}}.
}
So we have $f_t\leq (f_t)_{\max}\leq f_{\max}$ and hence $f_t$ satisfies Assumption \ref{Assump-barrier-f}. To show that $f_t$ satisfies Assumption  \ref{Assump-f}, we have
\eq{
(f_t)^{-\frac{1}{n-k}}= (1-t)(f_{\max})^{-\frac{1}{n-k}}+tf^{-\frac{1}{n-k}} \geq tf^{-\frac{1}{n-k}},
}
\eq{
D^2 f_t^{-\frac{1}{n-k}}= t D^2 f^{-\frac{1}{n-k}}, \quad Df_t^{-\frac{1}{n-k}}=t Df^{-\frac{1}{n-k}},
}
and
\eq{
\frac{|D f_t^{-\frac{1}{n-k}}|^2}{f_t^{-\frac{1}{n-k}}}= \frac{t^2 |Df^{-\frac{1}{n-k}}|^2}{(1-t)(f_{\max})^{-\frac{1}{n-k}}+tf^{-\frac{1}{n-k}}}\leq t \frac{|Df^{-\frac{1}{n-k}}|^2}{f^{-\frac{1}{n-k}}}.
}
Since $(f_t)_{\max}\leq f_{\max}$, we also have
\eq{
\frac{f_t^{-\frac{1}{n-k}}}{2+8((f_t)_{\max})^\frac{1}{n-k}}\geq t \frac{f^{-\frac{1}{n-k}}}{2+8 (f_{\max})^\frac{1}{n-k}}.
}
Thus, it is clear that $f_t$ satisfies (1)--(3) in Assumption \ref{Assump-f} for $-1\leq q<0$. The latter case $q\geq 0$ (i.e. $p\geq -k$) can be proved similarly. By using the full rank theorem (Theorem \ref{thm-full-rank}), we know that the solution $\varphi=\varphi_t$ to equation $\cL_t(\varphi)=0$ satisfies $A[\varphi_t]>0$. It is easy to verify that $f_t \geq f_{\min}$ for all $t\in [0,1]$. By the a priori estimates \eqref{eq-C0} and \eqref{eq-higher order derivative}, $\cL_t(\varphi)=0$ admits no solution on $\partial \cO_R$ if $R$ is sufficiently large for the case $-1\leq q\leq 1$. 

Next, we consider the remaining case $q>1$. As $f_{\min}\leq f_{t}\leq f_{\max}$, by a priori estimate \eqref{eq-higher order derivative} and equation $\cL_t(\varphi)=0$, there exists a constant $R_0>0$ such that 
\eq{
\|\varphi\|_{C^{4,\alpha}}<R_0, \quad p_{n-k}(A[\varphi])>\frac{1}{R_0}.
}
Suppose the claim is false. Then there always exists a sufficiently large constant $R\geq R_0$ and a solution to equation $\cL_t(\varphi)=0$ such that $\varphi \in \partial \cO_R$. Then we have
\eq{
\|\varphi\|_{C^{4,\alpha}}<R, \quad p_{n-k}(A[\varphi])>\frac{1}{R}.
}
So $\varphi\in \partial\cO_R$ implies that  $\varphi\geq \sqrt{\frac{q+1}{q-1}}$ and $\varphi(z_0)=\sqrt{\frac{q+1}{q-1}}$ is attained at some point $z_0\in \bbS^n$. At this point $z_0$, we have
\eq{
A[\varphi(z_0)] \geq \frac{1}{2}\( \sqrt{\frac{q+1}{q-1}}-\sqrt{\frac{q-1}{q+1}}\) \sigma,
}
and by Assumption \ref{Assump-barrier-f}, we deduce that
\eq{
\(\frac{(q+1)^{\frac{q+1}{2}}}{(q-1)^{\frac{q-1}{2}}}\)^{-1}&>(f_{\max})^{\frac{1}{n-k}}\geq ((f_t)_{\max})^{\frac{1}{n-k}}\\
&\geq f_t(z_0)^\frac{1}{n-k}=\varphi(z_0)^{-q} p_{n-k}^\frac{1}{n-k}(A[\varphi(z_0)])\\
&\geq \(\frac{(q+1)^{\frac{q+1}{2}}}{(q-1)^{\frac{q-1}{2}}}\)^{-1}.
}
This is a contradiction, and the claim follows. 

Recall that $\mathcal{L}_t$ is uniformly elliptic in $\mathcal{O}_R$, and hence the degree of $\mathcal{L}_t$ on $\mathcal{O}_R$ at $0$ is well-defined for all $t\in [0,1]$. By \cite[Prop. 2.2]{Li89}, the degree $\deg(\cL_t,\cO_R,0)$ is a homotopic invariant and thus 
\eq{\label{equality at t=1}
\deg(\cL_1,\cO_R,0)=\deg(\cL_0,\cO_R,0).
}
Hence, it suffices to compute the degree at $t=0$. By Lemma \ref{s6:lem-uniqueness}, the even, h-convex solutions to equation \eqref{s6:general-eq} are constant for any $q\geq -1$. Let $\varphi=c_0$ be a constant solution to equation $\mathcal{L}_0\(c_0 \)=0$, i.e.,
\eq{\label{s6:constant-solution}
\frac{1}{2} \( c_0-c_0^{-1}\)= c_0^q f_0^{\frac{1}{n-k}}=c_0^q (f_{\max})^{\frac{1}{n-k}}.
}
Note that Assumption \ref{Assump-barrier-f} on $f_{\max}$ guarantees the existence and uniqueness of the constant solution $\varphi=c_0$ in $\cO_R$. The linearized operator of $\cL_0$ at $\varphi=c_0$ is
\eq{
L_{c_0} \eta &=\frac{1}{n}\Delta_{\bbS^n} \eta+\( \frac{1}{2}+\frac{1}{2c_0^2}-q (f_{\max})^{\frac{1}{n-k}} c_0^{q-1} \)\eta\\
&=\frac{1}{n}\Delta_{\bbS^n} \eta+\( \frac{1-q}{2}+\frac{1+q}{2c_0^2} \)\eta.
}
Note that for each $q\geq -1$, we have 
\eq{
\frac{1-q}{2}+\frac{1+q}{2c_0^2}\leq 1.
}
Since the eigenvalues of the Beltrami-Laplace operator $\Delta_{\mathbb S^n}$ on $\mathbb S^n \subset \bbR^{n+1}$ are strictly less than $-n$ except for the first two eigenvalues $0$ and $-n$, and the coordinate functions of $\bbR^{n+1}$ span the eigenspace of eigenvalue $-n$ are odd, equation $L_{c_0} \eta=0$ only admits the unique even solution $\eta=0$. Thus the operator $L_{c_0}$ is invertible at $\varphi=c_0$. By \cite[Prop. 2.3]{Li89},
we have
\eq{\label{equiv-0-L_{c_0}}
\deg(\cL_0,\cO_R,0)=\deg(L_{c_0},\cO_R,0).
}
Moreover, there is only one positive eigenvalue $\frac{1-q}{2}+\frac{1+q}{2c_0^2}$ of $L_{c_0}$ with multiplicity $1$. 
By \cite[Prop. 2.4]{Li89}, we know that 
\eq{ \label{computation-of-degree-L_{c_0}}
\deg(L_{c_0},\cO_R,0)=-1.
}
Finally, it follows from \eqref{equality at t=1}, \eqref{equiv-0-L_{c_0}} and \eqref{computation-of-degree-L_{c_0}} that
\eq{
\deg(\cL_1,\cO_R,0)=-1,
}
which implies that there exists a $C^{4, \alpha }$-smooth, strictly h-convex, even solution $\varphi\in \cO_R$ to equation 
\eq{
p_{n-k}^\frac{1}{n-k}(A[\varphi])=\varphi^q f^{\frac{1}{n-k}}.
}
The regularity of $\varphi$ follows from Theorem \ref{thm-regularity-estimate}. This completes the proof of Theorem \ref{thm-horo-p-CM-problem}.
\end{proof}
\begin{remark}\label{remark-invertible}
Note that the first eigenvalue of $\Delta_{\bbS^n}$ for even functions are $-2(n+1)$, so for any $-4n-3 \leq q< -1$, by $c_0>1$ we have 
\eq{
1<\frac{1-q}{2}+\frac{1+q}{2c_0^2}<2(n+1).
}
Therefore, the operator $L_{c_0}$ is invertible at $\varphi=c_0$ for all $q\geq -4n-3$. In particular, for $n=1$, $L_{c_0}$ is invertible at $\varphi=c_0$ for all $q\geq -7$.
\end{remark}

\begin{proof}[Proof of Theorems \ref{thm-Lp-Weingarten-eq-hyperbolic}, \ref{thm-horo-p-Minkowski} and \ref{thm-horo-p-Minkowski-hyperbolic-plane}]
Let $n\geq 1$ and $0 \leq k\leq n-1$. We take 
\eq{
\cL_t(\varphi)=\(\frac{p_n(A[\varphi])}{p_k(A[\varphi])}\)^\frac{1}{n-k}-\varphi^q f_t^\frac{1}{n-k}.
}
If either $-1\leq q\leq 1$ for $n\geq 2$, or $-7\leq q\leq 1$ for $n=1$ (in this case, $k=0$), we define
\eq{
\cO_R:=\left\{ w \in \mathcal{B}_0^{4, \alpha}(\bbS^n) ~|~  \| w  \|_{C^{4, \alpha}}<R, \ w >1 +\frac{1}{R}, 
%\ A[w] > 0,
\ \frac{p_n(A [w])}{p_k(A [w])}  > \frac{1}{R}\right\};
}
If $q>1$ for $n\geq 1$, we define
\eq{
\cO_R:=\left\{ w \in \mathcal{B}_0^{4, \alpha}(\bbS^n) ~|~  \| w  \|_{C^{4, \alpha}}<R, \ w >\sqrt\frac{q+1}{q-1}, %\ A[w] > 0,
\  \frac{p_n(A [w])}{p_k(A [w])}  > \frac{1}{R} \right\}.
}
Then we have $\cL_t^{-1}(0)\cap \partial \cO_R=\emptyset$ if $R>0$ is sufficiently large. Here, in order to show that the solution to equation $\cL_t(\varphi)=0$ satisfies $A[\varphi]>0$, we use Lemma \ref{lem-elliptic-(p_n/p_{n-k})^{1/k}} instead of the full rank theorem (Theorem \ref{thm-full-rank}). For the case $n=1$ and $-7 \leq q \leq 1$, the uniqueness of the constant solution $\varphi=c_0$ follows from Lemma \ref{s6:lem-uniqueness}, and the operator $L_{c_0}$ is invertible at $\varphi = c_0$ when $q\geq -7$ by Remark \ref{remark-invertible}. The remaining proof of Theorems \ref{thm-Lp-Weingarten-eq-hyperbolic}, \ref{thm-horo-p-Minkowski} and \ref{thm-horo-p-Minkowski-hyperbolic-plane} follows from the similar argument as that of Theorem \ref{thm-horo-p-CM-problem}.
\end{proof}

\providecommand{\href}[2]{#2}

\end{document}